\documentclass{amsart}

\usepackage{datetime}
\usepackage{amsfonts}
\usepackage{amsmath}
\usepackage{amssymb}
\usepackage{amsthm}
\usepackage{eucal}
\usepackage{graphicx,color}

\definecolor{duckgreen}{rgb}{0.0,0.309,0.153}
\definecolor{badgerred}{rgb}{0.715,0.004,0.004}
\definecolor{burntorange}{rgb}{0.801,0.332,0.0}
\usepackage{hyperref}
\hypersetup{colorlinks,
            filecolor=black,
            linkcolor=badgerred,
            citecolor=duckgreen,
            urlcolor=burntorange,
            bookmarksopen=true}

\theoremstyle{plain}

\newtheorem{lemma}{Lemma}

\newtheorem*{main}{Main Theorem}
\theoremstyle{definition}

\newtheorem*{remark}{Remark}
\numberwithin{equation}{section}

\newcommand{\bry}{\mathfrak{B}}
\newcommand{\Dh}{\,e^{-\frac{\sigma^2}{4}}\,\mathrm{d}\sigma}
\newcommand{\ddt}[1]{\left.\frac{\partial}{\partial t}\right|_{#1}}

\newcommand{\dfdt}[2]{\left.\frac{\partial #1}{\partial t}\right|_{#2}}
\newcommand{\dfdtau}[2]{\left.\frac{\partial #1}{\partial \tau}\right|_{#2}}

\newcommand{\eh}{\approx}
\newcommand{\gk}{\gamma_k}
\newcommand{\hs}{\mathfrak{h}}
\newcommand{\ham}{\hat{A}^-}
\newcommand{\hap}{\hat{A}^+}
\newcommand{\hapm}{\hat{A}^\pm}
\newcommand{\hcQ}{ {\hat\cQ}}
\newcommand{\lh}{\left(}

\newcommand{\nh}{\|_{\mathfrak{h}}}
\newcommand{\params}{\{b_k,W\}}
\newcommand{\pd}[1]{\frac{\partial}{\partial #1}}
\newcommand{\pl}{\partial}
\newcommand{\rh}{\right)_{\mathfrak{h}}}
\newcommand{\tc}{\tau_{11}}
\newcommand{\tz}{\tilde{z}}

\newcommand{\ve}{\varepsilon}
\newcommand{\vh}{\check{v}}
\newcommand{\vp}{\varphi}
\newcommand{\zi}{z^{\rm int}}
\newcommand{\zp}{{\zeta}}
\newcommand{\zt}{z^{\rm tip}}

\DeclareMathOperator{\Rm}{Rm}
\newcommand{\cA}{\mathcal{A}}
\newcommand{\cB}{\mathcal{B}}
\newcommand{\cD}{\mathcal{D}}
\newcommand{\cE}{\mathcal{E}}

\newcommand{\cL}{\mathcal{L}}
\newcommand{\cO}{\mathcal{O}}
\newcommand{\cP}{\mathcal{P}}
\newcommand{\cQ}{\mathcal{Q}}
\newcommand{\cS}{\mathcal{S}}
\newcommand{\cT}{\mathcal{T}}
\newcommand{\rd}{\mathrm{d}}
\newcommand{\tA}{A_2}
\newcommand{\dA}{\hat A}
\newcommand{\hb}{B_8}
\newcommand{\Amin}{{A_2^{\rm min}}}
\newcommand{\taumin}{\tau_1}
\newcommand{\tzz}{\tilde\zeta}
\newcommand{\bt}{\bar\tau}

\newcommand{\R}{{\mathbb{R}}}

\begin{document}

\title{Degenerate neckpinches in Ricci flow}
\author{Sigurd B.~Angenent}
\address[Sigurd Angenent]{University of Wisconsin-Madison}
\email{angenent@math.wisc.edu}
\urladdr{http://www.math.wisc.edu/\symbol{126}angenent/}

\author{James Isenberg} \address[James Isenberg]{University of Oregon}
\email{isenberg@uoregon.edu}
\urladdr{http://www.uoregon.edu/\symbol{126}isenberg/}

\author{Dan Knopf} \address[Dan Knopf]{University of Texas at Austin}
\email{danknopf@math.utexas.edu}
\urladdr{http://www.ma.utexas.edu/users/danknopf}
\thanks{SBA thanks NSF for support in DMS-0705431.  JI thanks NSF for
    support in PHY-0968612.  DK thanks NSF for support in DMS-0545984.}

\begin{abstract}
  In earlier work \cite{AIK11}, we derived formal matched asymptotic
  profiles for families of Ricci flow solutions developing Type-II
  degenerate neckpinches.  In the present work, we prove that there
  do exist Ricci flow solutions that develop singularities modeled on
  each such profile. In particular, we show that for each positive
  integer $k\geq3$, there exist compact solutions in all dimensions
  $m\geq3$  that become singular at the rate $(T-t)^{-2+2/k}$.
\end{abstract}

\maketitle

\setcounter{tocdepth}{2}
\tableofcontents

\section{Introduction}

While the work of Gu and Zhu \cite{GZ08} establishes the existence of
Ricci flow solutions that form Type-II singularities, it tells us very
little regarding the details of the evolving geometries of such
solutions. The numerical work of Garfinkle with one of the authors
\cite{GI03, GI07}, together with the formal matched asymptotics
derived by all three authors \cite{AIK11}, strongly suggest what some
of these details might be, at least for those solutions
$(\cS^{n+1},g(t))$ that are rotationally symmetric and involve a
degenerate neckpinch. However, these works do not prove that solutions
with the prescribed behavior exist. In the present work, we prove that
indeed, for each of the prescribed models of the evolving geometry
near the singularity discussed in our work on matched asymptotics
\cite{AIK11}, there is a Ricci flow solution that asymptotically
approaches this model.  Specifically, we prove the main conjecture
stated in our earlier work \cite[\S7]{AIK11}. It follows from the
results proved here and stated below that in every dimension $n+1 \ge
3$ and for each integer $k\ge 3$, there is a Ricci flow solution that
develops a degenerate neckpinch singularity with the characteristic
profile (locally) of a Bryant soliton, and with the rate of curvature
blowup given by
\begin{equation}
  \label{rate}
  \sup_{x \in \cS^{n+1}}  |\Rm(x,t)|\sim\frac{C}{(T-t)^{2-2/k}},
\end{equation}
for $t$ approaching the singularity time $T<\infty$, and for some
constant $C$.

In determining that there exist Ricci flow solutions that form Type-II
singularities, Gu and Zhu \cite{GZ08} show that there are solutions in
which the curvature of the evolving metrics satisfies
$\limsup_{t\nearrow T}\{(T-t)\sup_{x\in\cS^{n+1}}|\Rm(x,t)|\}=\infty$.
Their work does not, however, determine the specific rate of curvature
blowup in these solutions.  Enders \cite{Enders07} has defined
\emph{Type-A} singularities as those with curvature blowing up at the
rate $(T-t)^{-r}$, with $r \in [1,\frac{3}{2})$, and has questioned
whether there are any compact solutions outside this class.  Our work
shows that indeed there are: as noted above, we prove that there exist
solutions with curvature blowing up at the rate $(T-t)^{-2+2/k}$ for
all positive integers $k \ge3$; all but the $k=3$ solutions fall
outside Enders' Type-A class of solutions.

Noting the existence of solutions with these discrete curvature blowup
rates, we are led to ask if there exist (compact) solutions ---
degenerate neckpinch or otherwise --- that exhibit blowup rates other
than these discrete values. Our work does not shed light on this
question.  We do note that for two-dimensional Ricci flow, in which
case the flow is conformal and the conformal factor evolves by
logarithmic fast diffusion, $u_t=\Delta\log u$, Daskalopoulos and
Hamilton \cite{DH04} prove that there exist complete noncompact
solutions that form singularities at the rate $(T-t)^{-2}$.

\medskip

The class of metrics we consider in this work are
$\rm{SO}(n+1)$-invariant Riemannian metrics on the sphere
$\cS^{n+1}$. We work with Ricci flow solutions $g(t)$ for such initial
data, and focus on solutions that develop a singularity at one or
both of the poles at finite time $T$ (proving that such solutions do
exist). To define what it means for a singularity to be a neckpinch,
we recall that a sequence $\{(x_j,t_j)\}_{j=0}^\infty$ of points and
times in a Ricci flow solution is called a \emph{blow-up sequence} at
time $T$ if $t_j \nearrow T$ and if $|\Rm(x_j,t_j)|\rightarrow\infty$
as $j\rightarrow\infty$; such a sequence has a corresponding
\emph{pointed singularity model} if the sequence of parabolic dilation
metrics $g_j(x,t):= |\Rm(x_j,t_j)|\,g\big(x, t_j+
|\Rm(x_j,t_j)|^{-1}\,t\big)$ has a complete smooth limit. We say that
a Ricci flow solution develops a \emph{neckpinch singularity} at time
$T$ if there is some blow-up sequence at $T$ whose corresponding
pointed singularity model exists and is given by the self-similarly
shrinking Ricci soliton on the cylinder $\mathbb{R}\times\cS^n$. We
call a neckpinch singularity \emph{nondegenerate} if every pointed
singularity model of any blowup sequence corresponding to $T$ is a
cylindrical solution, and we call it \emph{degenerate} if there is at
least one blowup sequence at $T$ with a pointed singularity model that
is not a cylindrical solution.

Rotationally symmetric nondegenerate neckpinches have been studied
extensively by Simon \cite{Simon00} and by two of the authors
\cite{AK04, AK07}. In the latter works, it is shown that there is an
open set of (rotationally symmetric) compact initial manifolds whose
Ricci flows develop nondegenerate neckpinch singularities, all of
which are Type-I in the sense that $\limsup_{t\nearrow
  T}\{(T-t)\sup_{x\in\cS^{n+1}}|\Rm(x,t)|\}<\infty$.  Further, in the
presence of reflection symmetry, \cite{AK07} provides a detailed set
of models for the asymptotic behavior of the geometry near a
developing nondegenerate neckpinch, with those models collectively
serving as attractors for these flows.

Unlike the formation of nondegenerate neckpinches, the formation of
degenerate neckpinches in Ricci flow is expected to be an unstable
property. This is evident in \cite{GI03, GI07} as well as in
\cite{GZ08}: in all three of these works, one studies flows that
develop degenerate neckpinches by considering one-parameter families
of initial data such that for all values of the parameter above a
threshold value, the Ricci flow solutions develop nondegenerate
neckpinches, while for all parameter values below that value, there is
no neckpinch singularity. The flows with initial data \emph{at} the
threshold value of the parameter are the ones that develop degenerate
neckpinches. This instability leads us to use the somewhat indirect
\emph{Wa\.{z}ewski retraction method} \cite{Wazewski47} to explore the
asymptotic behavior of the geometry near these degenerate neckpinches
as they form. We discuss this in detail below.

The matched asymptotics derived in \cite{AIK11} rely heavily on the
imposition of a series of \emph{Ansatz} conditions to characterize the
formal solutions of interest. (It is through these \emph{Ansatz}
conditions that one builds the formation of degenerate neckpinches
into the formal solutions.) By contrast, no such \emph{a priori}
assumptions are needed (or used) in the present work.  Rather, having
determined in \cite{AIK11} the nature and the explicit approximate
forms (in regions near the degenerate neckpinch) for our formal
solutions, we show here without imposing any further assumptions that
there exist Ricci flow solutions that approach each of the formal
solutions. It follows that degenerate neckpinches form in these Ricci
flow solutions.

Each of the solutions we consider here is characterized by an integer
$k \ge 3$. As our results show, there is at least a one-parameter
family of solutions corresponding to each value of $k$. In fact, our
construction in Section~\ref{InitialData} below reveals that degenerate
neckpinches form in solutions starting from a set of initial data of
codimension-$k$ in the space of $\rm{SO}(n+1)$-invariant solutions.
Besides determining the rate of curvature blowup, the integer $k$ also
characterizes to an extent the detailed asymptotic behavior of the
solution in a neighborhood of the singularity. While we discuss the
details of this characterization below (in
Section~\ref{sec:form-asympt-again}), we note here one important
feature that depends only on the parity of $k$. If $k$ is even, then
the solution is reflection symmetric across the equator, neckpinch
singularities occur simultaneously at both poles, and the volume of
$(\cS^{n+1},g(t))$ approaches zero at the time of the singularity. If
on the other hand $k$ is odd, then the neckpinch occurs at one pole
only, and the volume of $(\cS^{n+1},g(t))$ remains positive at the time
of the singularity. We note that in either case, for $t<T$, the
curvature has local maxima both at one of the poles and at a nearby
latitude sphere $\cS^n$ where the neck is maximally pinched. As $t$
approaches the singular time $T$, the distance between the neckpinch
sphere and the pole approaches zero, and the curvature becomes
infinite simultaneously (albeit at different rates) both at the
neckpinch sphere and at the pole.

A detailed statement of our main results depends on the details of the
asymptotic behavior of the formal model solutions. Referring to our
discussion of this behavior below in Section~\ref{sec:form-asympt-again},
we can state our main theorem as follows:

\begin{main} For every integer $k\geq3$ and every real number $b_k<0$,
  there exist rotationally symmetric Ricci flow solutions $(\cS^{n+1},g(t))$
  in each dimension $n+1\geq3$ that develop degenerate neckpinch singularities
  at $T<\infty$. For each choice of $n, k,$ and $b_k$, the corresponding
  solutions have distinct asymptotic behavior.

  In each case, the singularity is Type-II --- slowly forming --- with
  \[
  \sup_{x\in\cS^{n+1}}|\Rm(x,t)|=\frac{C}{(T-t)^{2-2/k}}
  \]
  attained at a pole, where $C=C(n,k,b_k)$.

  Rescaling a solution corresponding to $\{n,k,b_k\}$ so that the
  distance from the pole dilates at the Type-II rate
  $(T-t)^{-(1-1/k)}$, one finds that the metric converges uniformly on
  intervals of order $(T-t)^{1-1/k}$ to the steady Bryant soliton.

  Rescaling any solution so that the distance from the smallest neck
  dilates at the parabolic rate $(T-t)^{-1/2}$, one finds that the
  metric converges uniformly on intervals of order $\sqrt{T-t}$ to the
  shrinking cylinder soliton.

  Furthermore, the solutions exhibit the precise asymptotic behavior
  summarized in Section~\ref{sec:form-asympt-again}, and they satisfy
  the estimates summarized in Section~\ref{Outline}.
\end{main}

\medskip

Since the formal model solutions play a major role in this work, after
setting up the needed coordinates and metric representations for our
analysis in Section~\ref{ChooseCoordinatesWisely}, we carefully review
the nature of these formal solutions and their matched asymptotic
expansions in considerable detail in
Section~\ref{sec:form-asympt-again}. While the formal solutions we
discuss here are the same as those analyzed in \cite{AIK11}, we note
that here we use somewhat altered coordinate representations in
certain regions near the pole.  In Section~\ref{Structure}, we
describe the general structure of the proof of the main theorem and
outline its key steps.  The technical work to carry out these steps is
detailed in Sections~\ref{PuttingUpBarriers}--\ref{InitialData}.

\section{Coordinates for the four regions}
\label{ChooseCoordinatesWisely}
As noted above, in this paper we study ${\rm SO}(n+1)$-invariant
metrics $g$ evolving by Ricci flow on $\cS ^{n+1}\times[0,T_0)$, where
$T_0=T_0(g_0)\in(0,\infty]$ for initial data $g_0$. Each such metric
may be identified with functions
$\vp,\psi:(-1,1)\times[0,T_0)\rightarrow\R_+$ via
\begin{equation}
  g(x,t) = \vp^2(x,t)\,(\rd x)^2 + \psi^2(x,t)\,g_{{\rm can}},
  \label{MetricAnywhere}
\end{equation}
where $g_{{\rm can}}$ is the canonical round unit-radius metric on
$\cS^n$.  Smoothness at the poles requires that $\vp,\psi$ satisfy the
boundary condition~\eqref{BoundaryConditions} given below.  Under
Ricci flow, the quantities $\vp$ and $\psi$ evolve by
\begin{align*}
  \vp_t &= n\Big(\frac{\psi_{xx}}{\vp\psi}
  -\frac{\vp_x \psi_x}{\vp^2 \psi}\Big),\\
  \psi_t &= \frac{\psi_{xx}}{\vp^2} -\frac{\vp_x \psi_x}{\vp^3}
  +(n-1)\frac{\psi_x^2}{\vp^2 \psi} -\frac{n-1}{\psi},
\end{align*}
respectively. This system is only weakly parabolic, reflecting its
invariance under the full diffeomorphism group. Below, we remedy this
by suitable choices of gauge.

We are interested in proving the existence of solutions that are close
to the formal solutions constructed in \cite{AIK11}; accordingly, we
follow the notation of \cite{AIK11} in large measure.  That
paper describes solutions in four regions, the \emph{outer},
\emph{parabolic}, \emph{intermediate}, and \emph{tip}, using either of
two coordinate systems. We do the same here, although we choose the
alternate coordinate system to describe the intermediate and outer
regions. For clarity, we review both systems below.

\subsection{Coordinates for the parabolic region}
\label{ParabolicCoordinates}
We call the ${\rm SO}(n+1)$-orbit $\{0\}\times\cS ^n$ the ``equator''
and denote the signed metric distance from it by
\[
s(x,t):=\int_{0}^{x}\vp(\hat{x},t)\,\rd \hat{x}.
\]
Then the metric~\eqref{MetricAnywhere} may be written as
\begin{equation}
  g=(\rd  s)^2 +\psi^{2}(s(x,t),t)\,g_{{\rm can}},
  \label{MetricOut}
\end{equation}
where $\rd $ denotes the spatial differential. We write
\[
\ddt x \quad \text{and} \quad \ddt s
\]
to indicate time derivatives taken with $x$ and $s$ held fixed,
respectively.  With this convention, one has the commutator
\[
\left[\ddt x,\pd s\right] =-n\frac{\psi_{ss}}{\psi} \pd s
\]
and the relation
\[
\ddt x = \ddt s + nI[\psi] \pd s,
\]
where $I[\psi]$ is the nonlocal term
\[
I[\psi](s,t) := \int_0^s
\frac{\psi_{\hat{s}\hat{s}}(\hat{s},t)}{\psi(\hat{s},t)}\,\rd \hat{s}.
\]
In terms of these coordinates, the evolution of the
metric~\eqref{MetricOut} by Ricci flow is determined by the scalar
equation
\begin{equation}
  \dfdt \psi x =\psi_{ss}-(n-1)\frac{1-\psi_s^2}{\psi}.
  \label{psi-evolution}
\end{equation}
Smoothness at the poles requires that $\psi$ satisfy the boundary
conditions
\begin{equation}
  \left. \psi _s \right|_{x=\pm 1}=\mp 1.
  \label{BoundaryConditions}
\end{equation}
The quantity $\vp$, which is effectively suppressed in these
coordinates, evolves by
\[
\dfdt {(\log\vp)} x =n\frac{\psi_{ss}}{\psi}.
\]

\subsection{Coordinates for the other regions}
\label{IntermediateCoordinates}
We employ a different coordinate system in the other regions. In any
region where $\psi_s\neq0$ (e.g., in the neighborhood of the north
pole $x=1$ where $\psi_s<0$) one may use $\psi$ as a coordinate, writing
\begin{equation}
  g=z(\psi,t)^{-1} (\rd  \psi)^2 +\psi^{2}\,g_{{\rm can}},
  \label{MetricIn}
\end{equation}
where
\[
z(\psi(s,t),t):=\psi_s^2(s,t).
\label{Define-z}
\]
The evolution of the metric~\eqref{MetricIn} is determined by
\begin{equation}
  \dfdt z \psi =\cE _\psi[z],
  \label{eq:z-evolution}
\end{equation}
where $\cE _\psi$ is the quasilinear (but purely local) operator
\begin{equation}
  \cE _\psi [z] :=z z_{\psi\psi} -\frac{1}{2}z_\psi^2 +
  (n-1-z)\frac{z_{\psi}}{\psi} +2(n-1)\frac{(1-z)z}{\psi^2}.
  \label{EDefn}
\end{equation}
We can split the operator $\cE $ into a linear and a quadratic
component,
\[
\cE _\psi[z] = \cL _\psi[z] + \cQ _\psi[z],
\]
where
\begin{subequations} \label{eq:linear-and-quadratic-parts}
  \begin{align}
    \cL _\psi[z] &:=(n-1)\Bigl\{\frac{z_{\psi}}{\psi} +2\frac{z}{\psi^2}\Bigr\},\\
    \cQ _\psi[z] &:= z z_{\psi\psi} -\frac{1}{2}z_\psi^2
    -\frac{zz_{\psi}}{\psi} - 2(n-1)\frac{z^2}{\psi^2},
  \end{align}
\end{subequations}
respectively. The quadratic part also defines a symmetric bilinear
operator,
\begin{multline}
  \hcQ_\psi {[z_1,z_2]} := \frac 12 \bigl\{z_1 (z_2)_{\psi\psi} +
  z_2(z_1)_{\psi\psi} -(z_1)_{\psi}(z_2)_{\psi}\bigr\}
  \\
  - \frac12\frac{z_1(z_2)_{\psi} + z_2(z_1)_{\psi}}{\psi} -
  2(n-1)\frac{z_1z_2}{\psi^2}.
  \label{Q-bilinear-def}
\end{multline}
In terms of this notation, one has $\cQ _\psi[z] = \hcQ_\psi{[z,z]}$.

\section{The formal solution revisited}
\label{sec:form-asympt-again}

In \cite{AIK11}, we present a complete formal matched asymptotic
treatment of a class of rotationally symmetric Ricci flow formal
solutions that form degenerate neckpinches. These formal solutions
serve as the approximate models which the solutions we discuss here
asymptotically approach. Since we find it useful in our present
analysis to work with different coordinate representations (namely
$z(u,\tau)$) in the intermediate region than we use for the same
region in \cite{AIK11} (effectively $u(\sigma, t)$), we now briefly
review some of the analysis of \cite{AIK11}.

\subsection{Approximate solutions in the parabolic region}
\label{ApproximateParabolic}
Roughly speaking, the parabolic region is that portion of the
manifold, away from the tip, where the geometry approaches a shrinking
cylinder, and where the diameter of the neck has at least one local
minimum. (We give a precise definition of this region in
equation~\eqref{SpecifyParabolic} in Section~\ref{Outline} below.)
As discussed in \cite{AIK11}, it is useful in this region to work with
coordinates consistent with a parabolic cylindrical
blowup:\footnote{In Section~\ref{ChooseCoordinatesWisely}, $T_0>0$
  denotes the maximal existence time of a solution with initial data
  $g_0$. Here, $T>0$ denotes the singularity time of the formal
  solution we construct below, following \cite{AIK11}. In what follows,
  $T$ is fixed, albeit arbitrary and unspecified.}
\begin{equation}\label{DefineParabolicBlowup}
  u := \frac{\psi} {\sqrt{2(n-1)(T-t)}},\qquad \sigma
  :=\frac{s}{\sqrt{T-t}},\qquad \tau := -\log(T-t).
\end{equation}
In terms of these coordinates, the Ricci flow evolution equation
becomes\footnote{See equation~(3.4) of \cite{AIK11}.}
\begin{equation}
  \dfdtau u\sigma
  = u_{\sigma\sigma} -\left(\frac\sigma2 + nI[u]\right) u_\sigma
  + \frac 12 \Bigl(u-\frac 1u\Bigr) + (n-1)\frac{u_{\sigma}^{2}}{u},
  \label{u-evolution}
\end{equation}
with
\begin{equation}
  I[u](\sigma,\tau) = \int_0^{\sigma}\
  \frac{u_{\hat{\sigma}\hat{\sigma}}(\hat{\sigma},\tau)}{u(\hat{\sigma},\tau)}
  \,\rd \hat{\sigma}.
  \label{define-I}
\end{equation}

Setting $u=1+v$ and linearizing at $u=1$, we are led to
\[
\dfdtau v\sigma = v_{\sigma\sigma} - \frac\sigma2 v_\sigma + v
+\{\text{ nonlinear terms }\}.
\]
This form suggests writing the solution to this equation by expanding
$v$ in Hermite polynomials $h_m$, which are eigenfunctions of the
linear operator
\begin{equation}
  \cA :=\frac{\pl^{2}}{\pl\sigma^{2}}-\frac{\sigma}{2}\pd \sigma +1
  \label{define-A}
\end{equation}
satisfying $(\cA +\mu_m)h_m=0$, where
\begin{equation}
  \mu_m := \frac m2 -1.
  \label{eq:mu-m-defined}
\end{equation}
Writing $v$ in this fashion leads to the approximation
\begin{equation}
  u \eh 1 + \sum_{m=0}^\infty b_m e^{-\mu_m\tau} h_m(\sigma).
  \label{eq:inner-expansion}
\end{equation}
This expansion can at best be an approximation to the actual solution,
if only because the variable $\sigma$ is bounded. Note that one
characterization of the parabolic region is that it is the
(time-dependent) range of $\sigma$ for which the series terms in
equation~\eqref{eq:inner-expansion} are sufficiently small.

We are interested in solutions for which the term with $m=k$ is
dominant for some specified $k\geq 3$. (Compare \emph{Ansatz Condition
  2} of \cite{AIK11}.) Thus the approximation above takes the form
\[
u = 1 + b_k e^{-\mu_k\tau} h_k(\sigma) + \cdots,
\]
where the sign $b_k<0$ reflects the fact that the singularity forms at
the north pole.

We normalize so that the leading term in $h_k(\sigma)$ is $\sigma^k$.
We show below that, as they evolve, solutions of interest become
$C^1$-close to the formal solution
\begin{equation}
  u = 1 + b_k e^{-\mu_k\tau} \sigma^k + \cdots
  \label{eq:u-expansion-near-neck}
\end{equation}
in that portion of the parabolic region with $|\sigma|\gg1$.  Assuming
this for now, we compute an approximation of the formal solution in
terms of the quantity $z$ introduced above. To approximate $z$, we
differentiate the expression for $u$ with respect to $\sigma$, leading
to an expression for $z$ in terms of $\sigma$. Using the relation
between $u$ and $\sigma$, we then write $z$ in terms of $u$.  The
details of the calculation are as follows.

Since $\psi = u\sqrt{2(n-1)(T-t)}$ and since $s = \sigma\sqrt{T-t}$,
we have
\[
z = \psi_s^2 = 2(n-1) u_\sigma^2 \eh 2(n-1) k^2b_k^2
e^{-(k-2)\tau}\sigma^{2k-2}
\]
so long as $|\sigma|\gg1$. On the other hand, we also have
\[
\sigma \eh \Bigl\{ \frac{1-u} {-b_k} \Bigr\}^{1/k} e^{\gk\tau},
\]
where
\begin{equation}
\gk := \frac{\mu_k}{k} = \frac12 - \frac 1k.
\end{equation}
Thus for $|\sigma|\gg1$, we get
\begin{align*}
  z
  &\eh  2(n-1) k^2b_k^2 e^{-(k-2)\tau}\sigma^{2k-2} \\
  &\eh 2(n-1) k^2b_k^2 e^{-(k-2)\tau} \left(\Bigl\{\frac{1-u}
    {-b_k}\Bigr\}^{1/k} e^{\gk\tau}
  \right)^{2k-2}\\
  &= 2(n-1)k^2 (-b_k)^{2/k} e^{-2\gk\tau} (1-u)^{2-2/k}.
\end{align*}
Hence at the interface between the parabolic and intermediate regions,
where $u$ is slightly smaller than $1$, one has
\begin{equation}
  z \eh c_k e^{-2\gk\tau} (1-u)^{2-2/k},
  \label{eq:z-expansion-near-neck}
\end{equation}
with $c_k$ defined by
\begin{equation}
  c_k:=2(n-1)k^2 (-b_k)^{2/k}.
  \label{eq:c1-bk-relation}
\end{equation}

\subsection{Approximate solutions in the intermediate region}
\label{ApproximateIntermediate}
The intermediate region is a time-dependent subset of the neighborhood
of the north pole where $-1<u_\sigma<0$ and $0<u<1$. (We provide a
precise definition in \eqref{eq:intermediate-barrier-valid} below.) By
equation~\eqref{eq:z-evolution}, we know that
\[
\dfdt z \psi = \cE _\psi[z].
\]
Since $\psi = \sqrt{2(n-1)(T-t)}\,u = \sqrt{2(n-1)}e^{-\tau/2}u$, we
obtain
\begin{equation}
  \dfdtau z u = \frac{1}{2(n-1)}\cE _u[z] - \frac12 uz_u.
  \label{eq:z-u-evolution}
\end{equation}

Noting the expansion~\eqref{eq:z-expansion-near-neck} for $z$ near
$u=1$, our first impulse is to look for approximate solutions of the
form $z\eh e^{-2\gk\tau}Z_1(u)$.  Wishing to refine this
approximation, with the goal of constructing lower and upper barriers
$z_-\leq z\leq z_+$ for the intermediate region, we are led by past
experience to construct a formal Taylor expansion in time, namely
\begin{equation}
  z = e^{-2\gk\tau}Z_1(u) + e^{-4\gk\tau}Z_2(u) + \cdots
  = \sum_{m\geq 1} e^{-2m\gk\tau} Z_m(u).
  \label{eq:z-expansion}
\end{equation}
We substitute this expansion into \eqref{eq:z-u-evolution} and split
$\cE _u[z]$ into linear and quadratic parts, as in
\eqref{eq:linear-and-quadratic-parts}.  By comparing the coefficients
of $e^{-2m\gk\tau}$ in the resulting equation, we find that $Z_m$ must
satisfy
\[
-2m\gk Z_m = \frac{1} {2(n-1)} \left\{ \cL _u[Z_m] +
  \sum_{i=1}^{m-1}\hcQ_u {[Z_i, Z_{m-i}]} \right\} - \frac12
u\frac{dZ_m} {du}.
\]
This leads to the family of \textsc{ode}
\begin{equation}
  \frac12\bigl(u^{-1}-u\bigr)\frac{dZ_m} {du} +\bigl(u^{-2} + 2m\gk\bigr)Z_m
  = -\frac{1} {2(n-1)}\sum_{i=1}^{m-1}\hcQ_u {[Z_i, Z_{m-i}]},
  \label{eq:Zm-recursion}
\end{equation}
from which the $Z_m$ can be computed recursively, up to a solution of
the associated homogeneous equation (obtained by replacing the
\textsc{rhs} of equation~\eqref{eq:Zm-recursion} with zero). The
general solution of the $m^{\rm th}$ homogeneous equation is
\begin{equation}
  Z_{m,{\rm hom}}(u) = \hat c_m u^{-2} (1-u^2)^{1+2m\gk},
  \label{eq:Zm-homogeneous}
\end{equation}
where $\hat c_m$ is arbitrary.

For $m=1$, the \textsc{rhs} of equation~\eqref{eq:Zm-recursion}
vanishes. Therefore $Z_1$ satisfies the homogeneous equation,
whereupon setting $\hat c_1=c_k$ yields
\begin{equation}
  Z_1(u) = Z_{1,{\rm hom}}(u) = c_k u^{-2}(1-u^2)^{1+2\gk}.
  \label{eq:Z1}
\end{equation}
The function $Z_2$ is harder to compute (we have not tried to find a
solution in closed form).  However, the expression ~\eqref{eq:Z1} for
$Z_1$ shows that $Z_1 \eh c_k 2^{1+2\gk}(1-u)^{1+2\gk}$, ${Z_1}' \eh
-c_k(1+2\gk)2^{1+2\gk}(1-u)^{2\gk}$, and ${Z_1}'' \eh c_k
(\gk+2\gk^2)2^{2+2\gk}(1-u)^{2\gk-1}$ as $u\nearrow1$.  Thus we
calculate
\[
\hcQ_u {[Z_1, Z_1]} = \cQ _u [Z_1] \eh-c_k^2
2^{1+4\gk}(1-4\gk^2)(1-u)^{4\gk}.
\]
It follows that for some constant $\hat c$, the solution $Z_2$ of
equation~\eqref{eq:Zm-recursion} takes the form
\begin{equation}
  Z_2(u) = \bigl(\hat c+o(1)\bigr) (1-u)^{4\gk} \quad\text{as}\quad u\nearrow 1.
  \label{eq:Z2-asymptotics}
\end{equation}
The solution $Z_{2,{\rm hom}}$ of the homogeneous equation is
dominated by $(1-u)^{1+4\gk}$, hence is smaller than the \textsc{rhs}
of \eqref{eq:Z2-asymptotics} for $u\nearrow1$. Therefore, all
solutions $Z_2$ of equation~\eqref{eq:Zm-homogeneous} satisfy
\eqref{eq:Z2-asymptotics}.

It follows from these calculations that, near the neck, where $u\eh
1$, the first two terms in our formal expansion of $z$ predict that
\begin{equation}
  z = c_k e^{-2\gk\tau}u^{-2}(1-u^2)^{1+2\gk}
  + \cO \bigl(e^{-4\gk\tau}(1-u)^{4\gk}\bigr).
  \label{eq:1}
\end{equation}
Since $\gk<\frac12$, we have $1+2\gk > 4\gk$. Thus for fixed $\tau$,
the remainder term becomes larger than the leading term as $u\nearrow
1$. In \cite{AIK11}, we characterize the intermediate region using the
coordinate system~\eqref{MetricOut}; here, we identify it in terms of
the system~\eqref{MetricIn} as the region where our formal solution
$e^{-2\gk\tau}Z_1$ is a good approximation to the true solution. The
formal solution is a good approximation only so long as the remainder
is smaller than the leading term. This is true where
$e^{-4\gk\tau}(1-u)^{4\gk}\ll e^{-2\gk\tau}u^{-2}(1-u^2)^{1+2\gk}$, or
equivalently where $e^{-2\gk\tau} \ll (1-u)^{1-2\gk} = (1-u)^{2/k}$.
So our approximation in the intermediate region is valid where $1-u
\gg e^{-k\gk\tau}$.  Expansion~\eqref{eq:u-expansion-near-neck} in the
parabolic region shows that $1-u \gg e^{-k\gk\tau}=e^{-(k/2-1)\tau}$
precisely when $\sigma\gg 1$.

\medskip
Going in the other direction, as $u\searrow0$, one has the expansion
\begin{multline}
  \cQ _u [Z_1] \eh -2c_k^2(n-4)u^{-6} + 4c_k^2(n-3)(1+2\gk)u^{-4}\\
  -2c_k^2(1+2\gk)[n(1+4\gk)-10\gk-1]u^{-2}.
\end{multline}
Therefore, in general dimensions, the solution $Z_2$ of
equation~\eqref{eq:Zm-recursion} satisfies
\begin{equation}
  Z_2(u)=\cO  (u^{-4})\quad\text{as}\quad u\searrow 0.
\end{equation}
Near the tip, where $u\eh0$, the first two terms in our formal
expansion of $z$ predict that
\begin{equation}
  z=c_k e^{-2\gk\tau}u^{-2} + \cO  \bigl(e^{-4\gk\tau}u^{-4}\bigr).
  \label{eq:z-intemediate-near-tip}
\end{equation}
This shows that our approximation in the intermediate region is
expected to be valid in general only for $u\gg e^{-\gk\tau}$, or
equivalently for $\psi\gg e^{-(1-1/k)\tau}$, which is in agreement
with equation~(6.11) of \cite{AIK11}.

\subsection{Approximate solutions at the tip}
\label{approximate-sol-at-tip}

Since the solutions in the intermediate region are predicted to be
valid only for $u\gg e^{-\gk\tau}$, we introduce the variable
\[
r = e^{\gk\tau} u
\]
to be used in the tip region, which is defined in
equation~\eqref{eq:tip-barrier-valid} below.  Then, in view of
\eqref{eq:z-u-evolution}, one has
\[
\dfdtau zr = \dfdtau zu + \gk rz_r
\]
and hence
\[
\dfdtau zr = \frac{1}{2(n-1)} \cE _u[z] - \frac{r}{k}z_r
=\frac{e^{2\gk\tau}}{2(n-1)} \cE _r[z] - \frac{r}{k}z_r.
\]
We rewrite this as $\cT _r[z] = 0$, where
\begin{equation}
  \cT _r[z] :=
  e^{-2\gk\tau}\left\{ \dfdtau zr + \frac{r}{k}z_r\right\}
  - \frac{1}{2(n-1)} \cE _r[z].
  \label{eq:dz-dtau-r-constant}
\end{equation}
For large $\tau$, we assume for the formal argument here that we can
ignore the terms containing $e^{-2\gk\tau}$, which reduces the
equation to
\[
\cE _r[z]=0.
\]
The solutions of this equation (see \cite{ACK09,AIK11}) are the Bryant solitons,
\[
z(r) = \bry(a_k r),
\]
with $a_k>0$ an arbitrary constant (to be fixed below by matching considerations).
To obtain more accurate approximate solutions of $\cT _r[z]=0$, one could try an
expansion of the form
\begin{equation}
  z = \bry(a_k r) + e^{-2\gk\tau} \beta_1(r) + e^{-4\gk\tau}\beta_2(r) + \cdots.
  \label{eq:tip-expansion}
\end{equation}
In Section~\ref{sec:tip-region-sub-super}, we construct sub- and
supersolutions based on the first two terms of this expansion.  If as
in \cite{ACK09} and \cite{AIK11} we normalize the Bryant soliton by
requiring
\begin{equation}
  \bry(r) = \frac{1+o(1)}{r^2}, \qquad (r\to\infty),
  \label{eq:bryant-normalization}
\end{equation}
then the \emph{Ansatz}~\eqref{eq:tip-expansion} leads to
\[
z \sim a_k^{-2} r^{-2} \qquad (r\to\infty),
\]
and thus in terms of the $u$ coordinate, to
\[
z \sim a_k^{-2} e^{-2\gk\tau} u^{-2}
\qquad
(\tau\to\infty,\;u\text{ small}).
\]
Matching this expression with that given in
\eqref{eq:z-intemediate-near-tip}, we obtain $c_k=a_k^{-2}$, namely
\begin{equation}
  a_k := \frac{1}{\sqrt{c_k}}.
  \label{eq:A-c1-relation}
\end{equation}

\subsection{General features of the formal solutions}
\label{General}
In addition to recalling (with certain adaptations) the explicit
coordinate representations of the formal solutions in the parabolic,
intermediate, and tip regions, we wish to note a few general features
of the formal solutions that play a role in the statement and proof
of our main theorem. We first note that the
matched asymptotics analysis of \cite{AIK11} holds for all space
dimensions $n+1\ge 3$, for all Hermite indices $k\ge 3$, and for all
negative real values of the constant $b_k$. We thus observe that the
set of formal solutions is parameterized by these three numbers $\{n,
k, b_k\}.$

Next, it follows from the form of the formal solutions in the
parabolic region (and from the properties of the Hermite polynomials)
that the formal solutions are reflection-symmetric across the ($s=0$)
equator if and only if $k$ is even. These are the solutions that have
degenerate neckpinches forming simultaneously at both poles. Also,
asymptotically, they have only tip, intermediate, and parabolic
regions. Hence, in studying (below) the convergence of full Ricci flow
solutions to these formal approximations, it is only for the solutions
with odd $k$ that we must work in an outer region (beyond the
parabolic region) as well as in the other regions.

Finally, we note that if we calculate the curvature for one of these
formal solutions (see Section~6 of \cite{AIK11}), we find that the
norm of the curvature tensor achieves its maximum value at the tip,
where it takes the value
\begin{equation}
  |\Rm(\mathrm{tip},t)| =\frac{C_k}{(T-t)^{2-2/k}},
\end{equation}
with $C_k$ a constant\footnote{At the tip, the sectional curvatures
$\mathcal{K}^\top$ of a $2$-plane tangent to an $\cS^n$ factor and
$\mathcal{K}^\bot$ of an orthogonal $2$-plane both take the value
$\mathcal{K}=\frac{\hat C_k}{(T-t)^{2-2/k}}$, with
$\hat C_k := \frac{4}{k^2 (n-1)^2}(2 |b_k|)^{-\frac{2}{k}}.$}
depending only on $n, k,$ and $b_k$. The curvature of a
Ricci flow solution that asymptotically approaches this formal
solution necessarily blows up at the same rate.

\section{Outline of the proof}
\label{Structure}
The main result of this paper, as stated above in the main theorem,
is that for each formal degenerate neckpinch solution described above
(with specified values of the parameters $\{n, k, b_k\}$), there exist
rotationally symmetric Ricci flow solutions that asymptotically approach
this solution, and therefore share its asymptotic properties.  This result
verifies the conjecture in Section~7 of \cite{AIK11}. The rest of this paper
is devoted to proving this result. In this section, we outline the overall
strategy we use to carry out the proof.  In the remaining sections, we provide
the details.

\medskip

Our strategy for proving that there exist solutions obeying the
asymptotic profiles derived in \cite{AIK11} follows Wa\.{z}ewski's
retraction principle \cite{Wazewski47}, as used in \cite{AV97}. We
apply this to solutions of Ricci flow, regarded as trajectories
\[
g:[0,T_0) \to C^\infty \left(\big( T^*\cS^{n+1} \otimes_{\rm sym}
  T^*\cS ^{n+1} \big)_+\right)
\]
satisfying the isometry condition $\Upsilon^* g(t)=g(t)$ for all
$\Upsilon\in {\rm SO}(n+1)$ and $t\in[0,T_0)$. As noted in
Section~\ref{ChooseCoordinatesWisely}, such solutions may be naturally
identified with ordered pairs of maps
\[
(\vp,\psi):[0,T_0)\rightarrow\big(\R_+ \times \R_+ \big)
\]
satisfying the boundary conditions~\eqref{BoundaryConditions} for all
$t\in[0,T_0)$, where $T_0=T_0(g_0)$ is the maximal existence time of
the unique solution with initial data $g_0$.

To implement this strategy, we define tubular neighborhoods $\Xi_\ve$ of the
formal (approximate) solutions discussed above and in \cite{AIK11}. For a given
fixed formal solution $\hat g_{\{ n,k,b_k\}}(t)$ with singularity at time
$T$, and for a given positive $\ve \ll 1$, we choose a tubular neighborhood of
$\hat g_{\{n,k,b_k\}}(t)$ by specifying time-dependent inequalities to be
satisfied by $\vp$ and $\psi$, with these inequalities becoming more restrictive
as $t\nearrow T$, narrowing to the formal solution at time $t=T$.  Decomposing
the boundary of the tubular neighborhood
$\pl\Xi_\ve=(\pl\Xi_\ve)_-\cup(\pl\Xi_\ve)_\circ\cup(\pl\Xi_\ve)_+$, we
construct barriers and prove \emph{entrapment lemmas} which show that solutions
$g\in\Xi_\ve$ never contact $(\pl\Xi_\ve)_+$. We also prove an \emph{exit lemma}
which shows that any solution $g\in\Xi_\ve$ that contacts $(\pl\Xi_\ve)_-$ must
immediately exit.  This ensures that the exit times of those solutions that
never reach the neutral part $(\pl\Xi_\ve)_\circ$ of the boundary depend
continuously on the initial data.  In Section~\ref{InitialData}, we construct
(for each specified formal solution) a $k$-dimensional family of initial data
$g_\alpha(t_0) \in (\Xi_\ve\cap\{t=t_0\})$ parameterized by $\alpha\in {\cB}^k$
(the closure of a $k$-dimensional topological ball) such that all
$g_\alpha(t_0)$ with $\alpha\in\pl\cB ^k$ lie in the exit set $(\pl\Xi_\ve)_-$,
and such that the map $\phi:\alpha\in\pl\cB ^k\to g_\alpha(t_0) \in
(\pl\Xi_\ve)_-$ is not contractible. (As seen in Section~\ref{InitialData},
$t_0\in[0,T)$ is chosen to be very close to the singularity time $T$ of the
formal solution.) We choose the initial data $g_\alpha(t_0)$ at a sufficiently
large distance from the neutral boundary $(\pl\Xi_\ve)_\circ$ to guarantee that
as long as the corresponding solution $\{g_\alpha(t):t\geq t_0\}$ stays within
$\Xi_\ve$, it never reaches $(\pl\Xi_\ve)_\circ$.  In particular, the
$g_\alpha(t_0)$ are chosen so that the solutions $g_\alpha(t)$ can only leave
$\Xi_\ve$ through $(\pl\Xi_\ve)_-$.  If for every $\alpha\in\cB ^k$ the solution
$g_\alpha(t)$ were to leave $\Xi_\ve$ at some time $t_\alpha>t_0$, then the exit
map $\alpha\mapsto g_{\alpha}(t_\alpha)$ would yield a continuous map $\Phi: \cB
^k \to (\pl\Xi_\ve)_-$.  Since $t_\alpha=t_0$ for all $\alpha\in\pl\cB ^k$, the
map $\Phi$ would extend $\phi$, and therefore $\phi$ would have to be
contractible, in contradiction with our construction of $\phi$.  It follows that
for at least one $\alpha\in \cB^k$, the solution $g_\alpha(t)$ never leaves the
tubular neighborhood $\Xi_\ve$.  Such a solution must then exhibit the
asymptotic behavior derived in \cite{AIK11} (and summarized above in
Section~\ref{sec:form-asympt-again}) for the formal model solution $\hat
g_{\{n,k,b_k\}}(t)$.

\begin{figure}[t]
  \centering
  \includegraphics[width=\textwidth]{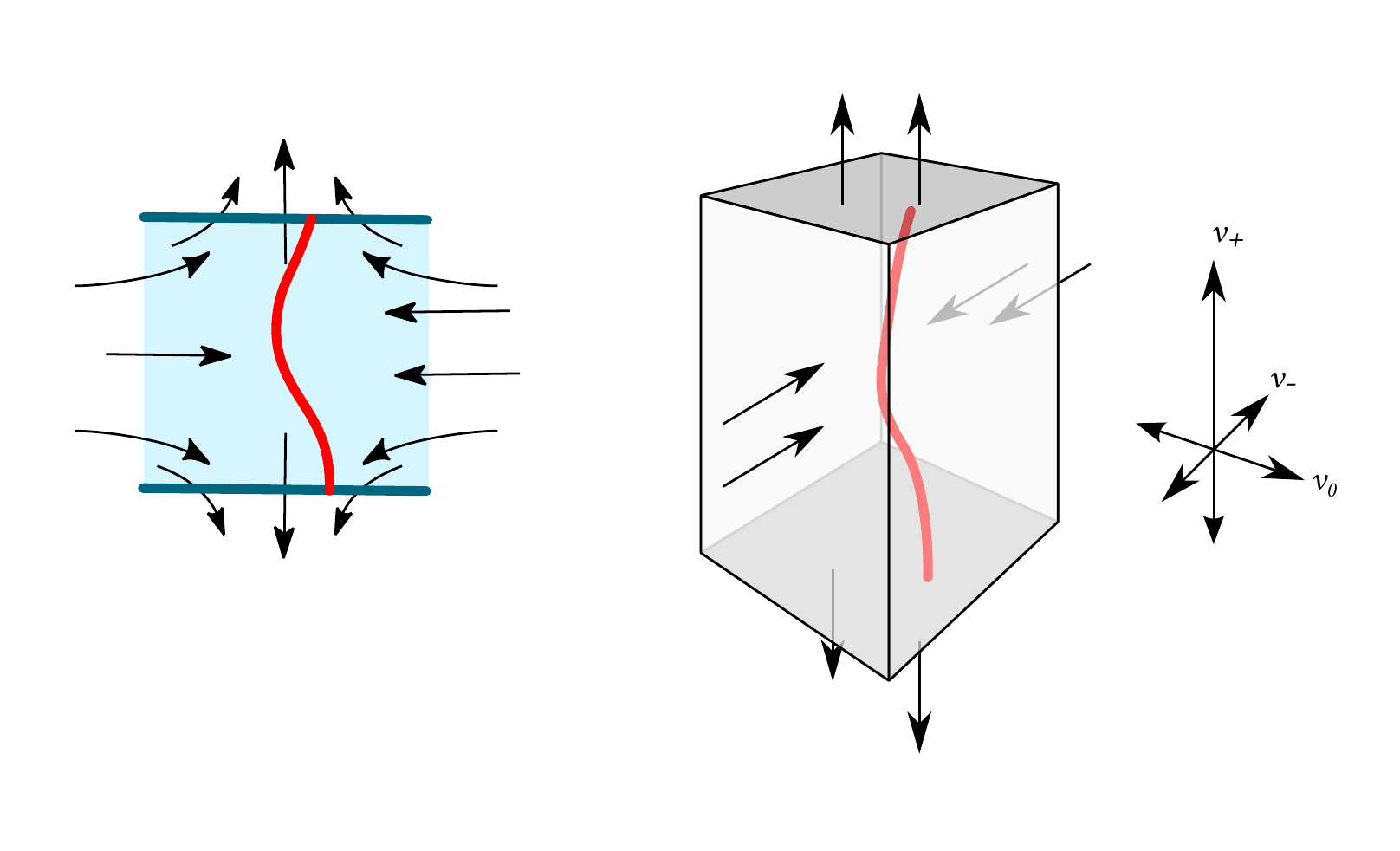}
  \caption{A schematic view of a tubular neighborhood $\Xi_\ve$.}
\end{figure}

We outline further the steps of this strategy in Section~\ref{Outline}
below.  Before doing so, we briefly describe how the tubular neighborhood
$\Xi_{\ve}$ (hereafter called ``the tube'') is defined in each of the four
regions. In subsequent sections, we show that the tube, as defined here and
(more explicitly) in Section~\ref{PuttingUpBarriers}, has all the properties
needed to carry through the proof.

\begin{figure}[ht]
  \centering
  \includegraphics[width=0.8\textwidth]{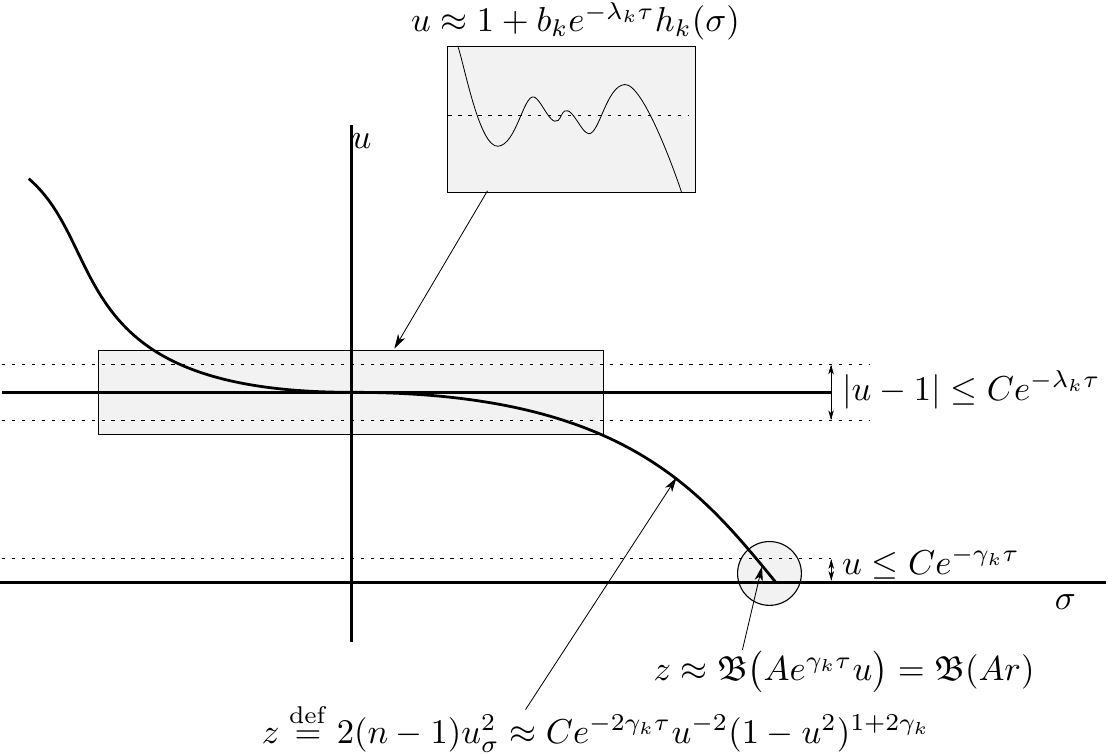}
  \caption{An overview of the separate regions in the solution }
  \label{fig:overview}
\end{figure}

\subsection{Defining the tube in the tip and intermediate regions}
In both of these regions, we define the tubular neighborhood
$\Xi_{\ve}$ using upper and lower barrier functions, $z_+(u,\tau)$ and
$z_-(u,\tau)$, which are defined with respect to the coordinate system
described in Section~\ref{IntermediateCoordinates}. By the parabolic
maximum principle, an evolving solution cannot cross a barrier. Thus
the barriers provide an \emph{entrapment condition}, ensuring that
solutions never exit through the portions of $(\pl\Xi_\ve)_+$
they define. The barriers are constructed for a precise subset of
$u\in[0,1)$ using functions adapted to each region separately.  To
complete the construction, we ``patch'' these functions together in
the transitional areas where the regions overlap, using the fact that
the maximum (minimum) of two subsolutions (supersolutions) is a
subsolution (supersolution).  This work is done in Section~\ref{PuttingUpBarriers}.

\subsection{Defining the tube in the parabolic region}
If $\Xi_\ve$ could be defined by barriers everywhere, it would follow
that solutions starting from an open set of rotationally symmetric
initial data would all develop degenerate neckpinches at the same time
$t=T$.  This is of course impossible, since one can always change the
singularity time of a solution by shifting it in time.  So rather than
constructing barriers in the parabolic region, we must adopt an
alternate (necessarily more complicated) approach.

We define $\Xi_\ve$ in the parabolic region working with the
coordinate system discussed in Section~\ref{ParabolicCoordinates}. In
order to do so, we must introduce more notation. As discussed in
Section~\ref{ApproximateParabolic} and in \cite{AIK11}, the operator
$\cA $ defined in~\eqref{define-A} naturally arises in the analysis of
the parabolic region, and we are led to write the formal solution as
an expansion in its Hermite polynomial eigenfunctions.  To make this
more precise, we note that $\cA $ is self-adjoint in the weighted
Hilbert space $\hs := L^2(\R; \Dh)$. Denoting the inner product in
$\hs$ by $\lh f,g \rh :=\int_{\R}f(\sigma)g(\sigma)\Dh$, and the norm
by $\|f\nh:= \surd{\lh f,f \rh}$, we find that the Hermite
polynomials $\{h_m(\cdot)\}_{m=0}^{\infty}$ constitute a complete
orthogonal basis of $\hs$.  We normalize so that
$h_m(\sigma)=\sigma^m+\cdots$.

Solutions $v=u-1$ do not belong to $\hs$, because they are not defined
for all $\sigma$.  So we fix $\eta>0$ to be a smooth even cutoff
function such that $\eta(y)=1$ for $0\leq |y|\leq1$ and $\eta(y)=0$
for $|y|\geq\frac65$.  We then define
\begin{equation}
\label{Support}
\vh(\sigma,\tau):=
\begin{cases}
  \eta\big(e^{-\gk\tau/4}\sigma\big)\,v(\sigma,\tau)
  &\text{for}\,|\sigma|\leq\frac65 e^{\gk\tau/5},\\
  0 &\text{for}\,|\sigma|>\frac65 e^{\gk\tau/5}.
\end{cases}
\end{equation}
It follows that $\vh\in\hs$ is smooth and satisfies
$\vh(\sigma,\tau)=v(\sigma,\tau)$ for $|\sigma|\leq e^{\gk\tau/5}$.
One computes from equation~\eqref{u-evolution} that
\begin{equation}
  \dfdtau\vh\sigma=\cA \vh + \eta N[v] + E[\eta,v],
  \label{CutoffEvolution}
\end{equation}
where
\begin{align*}
  N[v](\sigma,\tau)&:=\frac{2(n-1)v_{\sigma}^2-v^2}{2(1+v)}-nI[v]v_\sigma,\\
  E[\eta,v](\sigma,\tau)&:=\left(\eta_\tau-\eta_{\sigma\sigma}+\frac{\sigma}{2}\eta_\sigma\right)v
  -2\eta_\sigma v_\sigma,
\end{align*}%
with
\begin{gather*}
  \eta_{\sigma} = \frac{\pl \eta(e^{-\gk\tau/4}\sigma\big)}{\pl\sigma}
  =e^{-\gk\tau/4}\,\eta'\bigl(e^{-\gk\tau/4}\sigma\big),\qquad
  \eta_{\sigma\sigma}
  = e^{-\gk\tau/2}\,\eta''\bigl(e^{-\gk\tau/4}\sigma\big),\\
  \eta_\tau = -\frac{\gk}{4}\sigma e^{-\gk\tau/4}\,
  \eta'\bigl(e^{-\gk\tau/4}\sigma\big),
\end{gather*}
and where $I[v]=I[u]$ is defined in~\eqref{define-I}. We note that in calculating
$\eta N[v]$ and $E[\eta,v]$, one may take $v$ to be identically zero for any
$|\sigma|>\frac65 e^{\gk\tau/5}$ at which $v$ is not defined.

Because $\vh\in\hs$, it makes sense to define the projections
\begin{subequations}
  \label{DefineProjections}
  \begin{align}
    v_{k-} &:=\sum_{j=0}^{k-1}
    \frac{\lh \vh,h_j\rh}{\| h_j\nh^2}\,h_j,\\
    v_{k}   &:=\frac{\lh \vh,h_k\rh}{\| h_k\nh^2}\,h_k,\\
    v_{k+} &:=\sum_{j=k+1}^{\infty} \frac{\lh \vh,h_j\rh}{\|
      h_j\nh^2}\,h_j \qquad\left(=\vh-\big(v_{k-}+v_{k}\big)\right).
  \end{align}
\end{subequations}
With regard to the linear operator $\cA $, the projections $v_{k-}$,
$v_{k}$, and $v_{k+}$ represent rapidly-growing, neutral, and
rapidly-decaying perturbations, respectively, of the formal solution.
We show below that $v_{k-}$ is associated to the boundary
$(\pl\Xi_\ve)_-$, $v_k$ to the boundary $(\pl\Xi_\ve)_\circ$, and
$v_{k+}$ to the boundary $(\pl\Xi_\ve)_+$.

In the parabolic region, $\Xi_\ve$ is defined by four conditions: a
solution $g(\cdot)$ belongs to $\Xi_\ve$ if it satisfies the three
$L^2$ inequalities
\begin{subequations}
  \begin{align}
    \| v_{k-} \nh  &\leq\ve e^{-\mu_k\tau},\label{vkminusL2} \\
    \| v_{k}-b_ke^{-\mu_k\tau}h_k\nh &\leq\ve e^{-\mu_k\tau},\label{vkL2}\\
    \| v_{k+}\nh &\leq\ve e^{-\mu_k\tau},\label{vkplusL2}
    \intertext{together with the pointwise inequality}
    \sup_{|\sigma|\leq P}\big(|v|+|\partial_\sigma v|\big)&\leq W
    e^{-\mu_k\tau},\label{DefineParabolicPW}
  \end{align}
\end{subequations}
where $b_k$ is the parameter that selects a particular formal solution,
$P$ is a constant chosen below to define the range of the parabolic region,
$W$ is a large constant to be fixed below, and $\mu_k$ is defined
in \eqref{eq:mu-m-defined}.  We show below that the first inequality above
is an \emph{exit condition}, meaning that any solution that contacts
$\pl\Xi_\ve$ with $\| v_{k-} \nh = \ve e^{-\mu_k\tau}$ immediately exits
$\Xi_\ve$.  We further show that the remaining three inequalities are
\emph{entrapment conditions}, ensuring that solutions of interest never
contact the parts of $\pl\Xi_\ve$ that they define.  This work is done in
Section~\ref{ExitStrategy}.

\subsection{Defining the tube in the outer region}
\label{OuterTube}
If $k\geq4$ is even, we do not have to deal with an outer region when defining
$\Xi_\ve$: in this case, the formal solution described in Section~\ref{sec:form-asympt-again}
is reflection symmetric. So we encase it (outside the central parabolic region) in barriers
that are themselves reflection symmetric, placing such properly ordered and patched barriers
on either side of the parabolic region.\footnote{If $k\geq4$ is even, it is not necessary
that the initial data we construct in Section~\ref{InitialData} be reflection symmetric. For
such data to belong to $\Xi_\ve$, it suffices that they lie inside the reflection symmetric
barriers defined in Section~\ref{PuttingUpBarriers}, and that they satisfy estimates
\eqref{vkminusL2}--\eqref{DefineParabolicPW}.}

If $k\geq3$ is odd, on the other hand, we define $\Xi_\ve$ outside the parabolic region
with respect to the coordinate system described in Section~\ref{IntermediateCoordinates},
using barriers adapted to a precise subset of $u\in(1,e^{\tau/2})$. This work, which is
highly analogous to the constructions for the tip and intermediate regions, is also done
in Section~\ref{PuttingUpBarriers}.

\subsection{Summary of the main estimates used in the proof}
\label{Outline}
We presume now that we have fixed $n\geq2$ and $k\geq3$, and chosen $b_k<0$, thereby
selecting a formal solution $u(\sigma,\tau)=1+b_k e^{-\mu_k\tau}h_k(\sigma)+\cdots$
(corresponding to $\hat g_{\{n,k,b_k\}}(\tau)$) to serve as a model in the
parabolic region for the solutions whose existence we prove in this
paper. Two parameters remain to be determined: the small constant $\ve$
appearing in \eqref{vkminusL2}--\eqref{vkplusL2}, and the large constant
$W$ appearing in \eqref{DefineParabolicPW}. Carefully tracking which estimates
depend on which choices, we show below that at the end of the proof, we can
choose $W$ sufficiently large, depending only on $b_k$, such that our exit and
entrapment results, Lemmas~\ref{ExitLemma}--\ref{PWentrapment}, and hence the
Main Theorem, hold for all sufficiently small $\ve$, depending on $\params$.
Toward this end, we proceed as follows.


Using the relation $z=\psi_s^2 = 2(n-1)u_\sigma^2$, we employ properly
ordered and patched barriers constructed in Section~\ref{PuttingUpBarriers}
to argue in Section~\ref{ExitStrategy} that there exist exist constants
$A^\pm$, $B$, $\delta_1$, and $\tau_8$,\footnote{At several steps of the proofs
in Sections~\ref{PuttingUpBarriers}--\ref{ExitStrategy}, one has to restrict to
sufficiently large times $\tau$; we address this in Section~\ref{InitialData} by
constructing initial data that become singular sufficiently quickly.}
depending only on $b_k$, such that solutions $v=u-1$ in the tube $\Xi_\ve$ satisfy
\[
    \frac{\sqrt{A^-}}{k} <
    \big|\pl_\sigma\big\{(e^{\mu_k\tau}|v|)^{\frac1k}\big\}\big|
    <\frac{\sqrt{A^+}}{k}
\]
for $\tau\geq\tau_8$ and
$|\sigma|\in\left(\frac35 P,\delta_1 e^{\gk\tau}\right)$, where
\begin{equation}
  P:=2\left(\frac{B}{-b_k}\right)^{\frac1k}.
  \label{DefineP}
\end{equation}
The inequality above provides a precise sense in which solutions in
$\Xi_\ve$ resemble the formal solution for large $|\sigma|$, i.e., in
the interface between the parabolic and intermediate regions.
Accordingly, we define the parabolic region as
\begin{equation}
    \cP :=\{\sigma:|\sigma|\leq P\}.
    \label{SpecifyParabolic}
\end{equation}
For later use, we note that $P$ and the other constants introduced in
Section~\ref{PuttingUpBarriers} depend only on $b_k$ and are in particular
independent of $W$.

Further analyzing the parabolic--intermediate interface, we prove
in Section~\ref{ExitStrategy} that there exist $C_\pm$ depending on
$\params$ such that solutions in $\Xi_\ve$ satisfy
\[
    C_-|\sigma|^k \leq e^{\mu_k\tau}|v| \leq C_+|\sigma|^k
\]
for $|\sigma|\in\left(\frac35 P,\delta_1 e^{\gk\tau}\right)$
and $\tau\geq\tau_8$. Combining
this with the pointwise bound~\eqref{DefineParabolicPW}, we then
show that for solutions in $\Xi_\ve$, one has a pointwise estimate
\[
\big|(\pl_\tau-\cA )v\big|\leq C_1 W^2 e^{-2\mu_k\tau}
\]
in $\cP $ for $\tau\geq\tau_9$, where $C_1$ depends only on $b_k$,
and $\tau_9\geq\tau_8$ depends on $\params$. We next argue that
there exists $C_2$ depending only on $b_k$ such that one has an $L^2$ bound
\[
\|(\pl_\tau-\cA )\vh\nh \leq C_2 e^{-\frac32\mu_k\tau}
\]
for all $\tau\geq\tau_{10}$, if $\tau_{10}\geq\tau_9$ is chosen
sufficiently large, depending on $\params$.

To finish the proof, we apply our estimates for $\|(\pl_\tau-\cA)\vh\nh$
to the three $L^2$ inequalities~\eqref{vkminusL2}--\eqref{vkplusL2}
defining $\Xi_\ve$ in the parabolic region to prove an exit lemma for
the first inequality and entrapment lemmas for the remaining two. Then
we use the analytic semigroup generated by $\cA $ to derive $C^1$ bounds
\[
\sup_{|\sigma|\leq P}\big\{|\vh(\sigma,\tau)| +
|\vh_\sigma(\sigma,\tau)|\big\} \leq C_3 e^{-\mu_k\tau}
\]
that hold at all sufficiently large times $\tau\geq\tau_{11}\geq\tau_{10}$, where
$\tau_{11}$ depends on $\params$, but $C_3$ depends only on $b_k$. We use this estimate
to prove an entrapment lemma that says that for $W$ large enough and all small enough $\ve$,
solutions in $\Xi_\ve$ never contact $\pl\Xi_\ve$ by achieving equality in~\eqref{DefineParabolicPW}.
Collectively, these results hold for all times $\tau\geq\bt\geq\tau_{11}$, thereby establishing
all necessary properties of the tubular neighborhood $\Xi_\ve$. The details of the steps outlined
here for the parabolic region are provided in Section~\ref{ExitStrategy}.

\medskip

The argument we have outlined above proves that for at least
one initial datum belonging to $\Xi_\ve$ at the initial time,
the solution that flows from this datum develops a degenerate
neckpinch --- provided that the set of such initial data is
nonempty. The work to construct such data is done in
Section~\ref{InitialData}, where we choose $t_0\in[0,T)$ close enough
to $T$ so that $t=t_0$ corresponds to $\tau=\bt$.  This completes the
outline of the proof of the main theorem, which validates the predictions
we made in \cite{AIK11}.\footnote{One may translate between the
  coordinate representations used in \cite{AIK11} and those used here
  by following the methods illustrated in
  Section~\ref{sec:form-asympt-again} above.}
We proceed in the sections below to supply the details of the proof.

\section{Constructing barriers}\label{PuttingUpBarriers}

\subsection{Barriers in the intermediate region}

In the intermediate region, one says that $z$ is a subsolution
(supersolution) of equation~\eqref{eq:z-u-evolution} if $\cD _u z
\leq0$~$(\geq0)$, where
\begin{equation}
  \cD _u z := \dfdtau z u -\frac12 (u^{-1}-u)z_u -
  u^{-2}z - \frac{\cQ _u[z]}{2(n-1)}.
  \label{eq:Du-definition}
\end{equation}
A lower (upper) barrier is a subsolution (supersolution) that lies
below (above) a formal solution in an appropriate space-time region.
These are obtained by first constructing parameterized families of
sub- and supersolutions, and then making suitable parameter choices to
ensure properly ordered barriers.

\begin{lemma}
  \label{lem:intermediate-barrier}
  Let $Z_1(u) = A_1 u^{-2}(1-u^2)^{1+2\gk}$ for some $A_1>0$.  There
  exist a function $\zp : (0, 1)\to\R$, a constant $B_1>0$, and a
  constant $\Amin < \infty$ depending only on $A_1$ such that for any
  $\tA\geq\Amin$, the function
  \begin{equation}
    z_\pm(u, \tau) := e^{-2\gk\tau}Z_1(u) \pm \tA e^{-4\gk\tau}\zp(u)
    \label{eq:barrier-ansatz}
  \end{equation}
  is a supersolution $(+)$ or a subsolution $(-)$ in the space-time
  region
  \begin{equation}
    B_1\sqrt{\frac{\tA}{A_1}}e^{-\gk\tau} \leq
    u\leq
    1-B_1\left( \frac{\tA}{A_1} \right)^{k/2} e^{-\mu_k\tau},  \qquad
    \tau \geq \taumin,
  \end{equation}
  where $\taumin$ depends only on $A_1$ and $\tA$.
\end{lemma}
We do not claim that $\zp(u)>0$, so the subsolutions and supersolutions provided
by this lemma may not be ordered.  In Lemma~\ref{OrderIntermediate}, below, we
construct properly ordered barriers $z_\pm$ by specifying a pair of constants
$A_1^{+}$ and $A_1^{-}$, setting $A_1=A_1^{+}$ in the expression
\eqref{eq:barrier-ansatz} for $z_{+}$ and $A_1=A_1^{-}$ in the expression
\eqref{eq:barrier-ansatz} for $z_{-}$. The time-dependent
domain~\eqref{eq:intermediate-barrier-valid} where these barriers are valid
provides the precise characterization of the intermediate region, as promised in
Section~\ref{ApproximateIntermediate} above.

\begin{proof}
  We verify here that $z_+$ is a supersolution.  (The verification
  that $z_-$ is a subsolution follows from the same argument, with the
  sign of $\tA$ changed.)
  Applying the operator $\cD _u$ defined in \eqref{eq:Du-definition}
  to the expression~\eqref{eq:barrier-ansatz}, we obtain
  \begin{multline*}
    \cD _u[z_+] = e^{-2\gk\tau}\Bigl\{ -\frac12\bigl( u^{-1} - u
    \bigr){Z_1}' -\bigl( u^{-2} + 2\gk \bigr) Z_1
    \Bigr\}\\
    +e^{-4\gk\tau} \Bigl\{ - \frac12\bigl( u^{-1} - u \bigr)\tA\zp' -
    \bigl(u^{-2}+4\gk\bigr)\tA\zp
    -\frac{1}{2(n-1)}\cQ _u[Z_1]\\
    \qquad - \frac{\tA}{n-1} e^{-2\gk\tau} \hcQ_u{[Z_1, \zp]}
    -\frac{\tA^2}{2(n-1)} e^{-4\gk\tau} \cQ _u[\zp] \Bigr\}.
  \end{multline*}
  It follows from the definition of $Z_1$ that the terms multiplying
  $e^{-2\gk\tau}$ vanish; thus we have
  \begin{multline}
    e^{4\gk\tau} \cD _u[z_+] = \tA \Bigl\{ -\frac12\bigl( u^{-1} - u
    \bigr)\zp' - \bigl(u^{-2}+4\gk\bigr)\zp \Bigr\}
    - \frac{1}{2(n-1)}\cQ _u[Z_1]\\
    - \frac{\tA}{n-1} e^{-2\gk\tau} \hcQ_u{[Z_1, \zp]} -
    \frac{\tA^2}{2(n-1)} e^{-4\gk\tau} \cQ _u[\zp].
    \label{eq:z-substituted-in-Du}
  \end{multline}
  We want to choose $\zeta$ and $\tA$ so that for sufficiently large
  values of $\tau$, $\cD _u[z_+]\geq 0$ ($\cD _u[z_-]\leq 0$ for the
  subsolution).  To achieve this, we first note that as
  $\tau\to\infty$, only the first two terms on the right in
  \eqref{eq:z-substituted-in-Du} remain.  We therefore choose $\zp$
  so that those first two terms together are positive so long as $\tA$
  is large enough. This then implies that for any fixed $u\in(0,1)$,
  one has $\cD_u[z_+](u,\tau)>0$ if $\tau$ is large enough.  We
  conclude the proof by estimating how large $\tau$ must be for the
  constant terms in \eqref{eq:z-substituted-in-Du} to outweigh the
  remaining terms.

\subsubsection*{Estimate for $\cQ_u[Z_1]$:}
For $0<u<1$ (a necessary condition for the intermediate region), it is
easy to see that
\[
|{Z_1}'| \leq \frac{C} {u(1-u^2)} Z_1 \quad\text{ and }\quad
|{Z_1}''|\leq \frac{C}{u^2(1-u^2)^2}Z_1.
\]
Using $\cQ _u[Z_1] = Z_1{Z_1}'' - \frac12 ({Z_1}')^2 -u^{-1}
Z_1{Z_1}'-2(n-1) u^{-2} Z_1^2$, we find that
\begin{equation}
  \bigl|\cQ _u[Z_1]\bigr|
  \leq CA_1^2 u^{-6}(1-u^2)^{4\gk}.
  \label{eq:QuZ1-estimate}
\end{equation}
Here and in what follows, unnumbered constants $C$ are generic and may
change from line to line.

\subsubsection*{Choice of $\zp$:}
Let $\zp : (0,1)\to\R $ be any solution of the inhomogeneous
\textsc{ode}
\begin{equation}
  -\frac12  \bigl( u^{-1} - u \bigr)
  \frac{\rd\zeta} {\rd u} - \bigl(4\gk+u^{-2}\bigr)\zeta
  = u^{-6}(1-u^2)^{4\gk}.
  \label{eq:zetap}
\end{equation}
Replacing $u$ by $u^2$ as the independent variable, one can rewrite
this equation as
\begin{equation}
  \frac{\rd\zeta}{\rd(u^2)} +
  \Bigl\{
  \frac{1+ 4\gk }{1-u^2} + \frac{1}{u^2}
  \Bigr\}
  \zeta=
  - \frac{(1-u^2)^{4\gk-1}}{u^6},
  \label{eq:zetap-regsing}
\end{equation}
whereupon it is easy to see that the equation has regular singular
points both at $u=0$ and at $u=1$. The solution of the corresponding
homogeneous equation is given by \eqref{eq:Zm-homogeneous} with $m=2$,
namely $\zeta_{\rm hom} = \hat c_2 u^{-2}(1-u^2)^{1+4\gk}$. One can either
use the Frobenius method of power series or variation of constants to
conclude that every choice of $\zp$ solving \eqref{eq:zetap} has the
asymptotic behavior
\begin{equation}
  \zp(u) =\left\{
    \begin{aligned}
      u^{-4} &+ \cO(u^{-2}\log u) & (u\searrow0),\\
      -(1-u^2)^{4\gk} &+\cO\bigl((1-u^2)^{4\gk+1}\log (1-u^2)\bigr) &
      (u\nearrow1)
    \end{aligned}\right.
  \label{eq:zp-asymptotics}
\end{equation}
at the endpoints of the interval $0<u<1$.

\subsubsection*{Constructing supersolutions:}
Estimate \eqref{eq:QuZ1-estimate} and formula~\eqref{eq:z-substituted-in-Du}
together tell us that
\begin{multline*}
  e^{4\gk\tau} \cD _u[z_+] \geq
  \left\{\tA-CA_1^2\right\}u^{-6}(1-u)^{4\gk}\\
  -\frac{\tA}{n-1} e^{-2\gk\tau}|\hcQ_u{[Z_1, \zp]}|
  -\frac{\tA^2}{2(n-1)}e^{-4\gk\tau}|\cQ _u[\zp]|.
\end{multline*}
To ensure that the $\tau$-independent term is positive, we must choose
$\tA > CA_1^2$.  We set
\[
\Amin = 2CA_1^2,
\]
which then implies that for all $\tA \geq \Amin$, one has
\begin{multline}
  e^{4\gk\tau} \cD _u[z_+] \geq\\
  CA_1^2 u^{-6}(1-u)^{4\gk} -\frac{\tA}{n-1}
  e^{-2\gk\tau}|\hcQ_u{[Z_1, \zp]}|
  -\frac{\tA^2}{2(n-1)}e^{-4\gk\tau}|\cQ _u[\zp]|.
\end{multline}

\subsubsection*{Determining where $z_+$ is a supersolution:}
By definition, one has
\[
2\hcQ_u{[Z_1, \zeta]} = Z_1\zeta''+ {Z_1}''\zeta - {Z_1}'\zeta'
-u^{-1}\bigl(Z_1\zeta'+{Z_1}'\zeta\bigr) -4(n-1)u^{-2}Z_1\zeta.
\]
For $0<u\leq\frac12$, we recall \eqref{eq:zp-asymptotics} to see that
\[
|\hcQ_u{[Z_1, \zp]}| \leq C A_1 u^{-8}
\]
and
\[
|\cQ _u[\zp]|\leq C u^{-10}.
\]
Combining these estimates, we find there is $c>0$ such that
\begin{align*}
  e^{4\gk\tau}\cD _u[z_+] & \geq c u^{-6} \bigl( A_1^2- CA_1\tA
  e^{-2\gk\tau}u^{-2} - C\tA^2 e^{-4\gk\tau}u^{-4}
  \bigr) \\
  & = cu^{-6}A_1^2 \Bigl( 1 - C \frac{\tA}{A_1(e^{\gk\tau}u)^2} - C
  \frac{\tA^2}{A_1^2(e^{\gk\tau}u)^4} \Bigr)
\end{align*}
for $0<u\leq \frac12$.  It follows from this inequality that there
exists a constant $B_1<\infty$ such that $z_+$ is a supersolution in
the region $(e^{\gk\tau}u)^2 \geq B_1^2 \tA/A_1$, equivalently
\[
B_1 \sqrt{\frac{\tA}{A_1}} e^{-\gk\tau} \leq u\leq \frac12.
\]

On the other hand, for the case $\frac12\leq u <1$, we have
$Z_1\sim A_1(1-u^2)^{1+2\gk}$ and $\zp\sim (1-u^2)^{4\gk}$, so that
\[
|\hcQ_u{[Z_1, \zp]}|\leq CA_1 (1-u^2)^{-1+6\gk}
\]
and
\[
|\cQ _u[\zp]| \leq C (1-u^2)^{-2+8\gk}.
\]
Thus in the interval $\frac12\leq u <1$, we have
\begin{multline*}
  e^{4\gk\tau}\cD_u[z_+] \geq\\
  c(1-u^2)^{4\gk} \bigl\{ A_1^2-CA_1\tA e^{-2\gk\tau}(1-u^2)^{-1+2\gk}
  - C\tA^2 e^{-4\gk\tau}(1-u^2)^{-2+4\gk} \bigr\}.
\end{multline*}
Hence, making $B_1$ larger if necessary, we see that $z_+$ is a
supersolution if
\[
1-u^2 \geq B_1 \left( \frac{\tA}{A_1} e^{-2\gk\tau}
\right)^{1/(1-2\gk)} = B_1 \left( \frac{\tA}{A_1} \right)^{k/2}
e^{-\mu_k\tau}.
\]
Since $1-u^2=(1-u)(1+u)$, we have $1-u<1-u^2<2(1-u)$, which proves the
lemma for $z_+$. The same arguments apply to $z_-$, so the proof is complete.
\end{proof}

We now construct a pair of properly ordered barriers, based on the sub- and
supersolutions of Lemma~\ref{lem:intermediate-barrier}. To do this, we first specify
\begin{equation}
A_1^\pm:=(1\pm\delta)c_k
\label{A1+-}
\end{equation}
where $c_k$ is defined in \eqref{eq:c1-bk-relation}, and $0<\delta\ll\frac12$ is a small
constant fixed here once and for all. We then define $A_3:=\max\{\tA(A_1^-),\tA(A_1^+)\}$,
where the dependence of $\tA$ on $A_1^+$ and $A_1^-$ is that which is discussed in
Lemma~\ref{lem:intermediate-barrier}, with the added requirement that
$\tA(A_1^\pm)\geq\Amin(A_1^\pm)$ are large enough to satisfy condition~\eqref{ReduceDependence}
in Lemma~\ref{PastingLemma} below. (This is necessary so that the ordered barriers we construct
here are compatible with those we construct in Section~\ref{sec:tip-region-sub-super} for the
tip region.) We further fix $\tau_2:=\max\{\taumin(A_1^-),\taumin(A_1^+)\}$
and $B_2:=\max\{B_1(A_1^-),B_1(A_1^+)\}$. We then have the following:

\begin{lemma}
\label{OrderIntermediate}
There exist $\tau_3\geq\tau_2$ and $B_3\geq B_2$ depending only on $b_k$, such that if $\zp$
is the function appearing in Lemma~\ref{lem:intermediate-barrier}, then
\begin{equation}
   z_\pm(u,\tau) = (1\pm\delta)c_k e^{-2\gk\tau}u^{-2}(1-u^2)^{1+2\gk}
    \pm A_3 e^{-4\gk\tau}\zp(u)
\end{equation}
are a pair of properly ordered barriers in the space-time region
\begin{equation}
    B_2\sqrt{\frac{2A_3}{c_k}}e^{-\gk\tau}
    \leq u\leq
    1-B_3\left( \frac{2A_3}{c_k} \right)^{k/2} e^{-\mu_k\tau},
    \qquad
    \tau\geq\tau_3.
    \label{eq:intermediate-barrier-valid}
\end{equation}
\end{lemma}

\begin{proof}
Using the asymptotic expansion~\eqref{eq:zp-asymptotics} for $\zp$,
we easily determine that $z_- < z_+$ as $u\searrow0$. It is also
straightforward to verify that on any interval $c\leq u\leq 1-c$,
there exists $\tau_3\geq\tau_2$ such that $z_- < z_+$ if
$c\leq u\leq 1-c$ and $\tau\geq\tau_3$.

To study the rest of the intermediate region, we
define $w:=1-u^2$. Then as $u\nearrow1$, the asymptotic
expansion~\eqref{eq:zp-asymptotics} implies that
\[
    w^{-1-2\gk}(z_+ - z_-) \geq
    2\delta c_k u^{-2}
    -2A_3 e^{-4\gk\tau} w^{2\gk-1}(1+w\log w).
\]
Using the facts that $2\gk-1=-2/k$ and that $|w\log w|\leq 1$
for $0<w\leq1$, one then estimates
for $u\leq1-B_3(2A_3 /c_k)^{k/2}e^{-\mu_k\tau}$
and $\tau\geq\tau_3$ that
\begin{align*}
    w^{-1-2\gk}(z_+ - z_-)
    &\geq 2\delta c_k u^{-2}
    -4A_3 e^{-4\gk\tau} w^{-2/k}\\
    &\geq 2c_k
    \big\{\delta u^{-2}
    -B_3^{-2/k}(1+u)^{-2/k}e^{-2\gk\tau_3}\big\}\\
    &>0
\end{align*}
provided that $B_3\geq B_2$ is chosen sufficiently large.
\end{proof}

\subsection{Barriers in the tip region}
\label{sec:tip-region-sub-super}
In Section~\ref{approximate-sol-at-tip}, we observe that the evolution
of the metric is governed by the parabolic \textsc{pde} $\cT _r[z] =
0$, where $r = e^{\gk\tau} u$, and where $\cT _r$ is defined in
equation~\eqref{eq:dz-dtau-r-constant}.
This motivates us to look for sub- and supersolutions --- and
then properly ordered upper and lower barriers --- of the form
$z_\pm=\bry(A_4 r)\pm e^{-2\gk\tau} \beta(r)$, where $\beta(r)$ is
to be chosen to solve a suitable \textsc{ode}, where $A_4$ is a
parameter to be determined, and where we recall that $\bry$ denotes
the functional expression for the Bryant soliton.

\begin{lemma}\label{lem:tip-barrier}
  For any $A_4>0$, there exist a bounded function
  $\beta:(0,\infty)\to\R$, a sufficiently small $B_4>0$, and a
  sufficiently large $\tau_4<\infty$, all depending only on $A_4$,
  such that
  \begin{equation}
    z_\pm := \bry(A_4r) \pm e^{-2\gk\tau} \beta(r)
    \label{eq:tip-barrier}
  \end{equation}
  are sub- $(z_-)$ and super- $(z_+)$ solutions in the region
  \begin{equation}
    0\leq r\leq B_4e^{\gk\tau}
    \label{eq:tip-barrier-valid}
  \end{equation}
  for all $\tau\geq\tau_4$.
\end{lemma}

We note that inequality~\eqref{eq:tip-barrier-valid} provides a working
definition of the tip region, which overlaps the intermediate
region~\eqref{eq:intermediate-barrier-valid} for all sufficiently
large times $\tau$.

As with the intermediate region, we do not claim that $\beta$ has a sign; to get a pair of
properly ordered barriers for the tip region, one makes suitable choices of the parameter
$A_4$ (namely $A_4^-$ and $A_4^+$ as defined in \eqref{OrderParameters} below),
relying on the fact that $\bry$ is monotone decreasing.

\begin{proof}
  Using the abbreviation $\hat B(r) := \bry(A_4r)$, we see that for the
  function $z=\hat B(r)+e^{-2\gk\tau}\beta(r)$ to be a a supersolution, it
  suffices that $\cT _r[z]\geq 0$. Using the definition of $\cT _r$ in
  \eqref{eq:dz-dtau-r-constant} and the definitions of $\cL _r$ and
  $\hcQ_r$ from \eqref{eq:linear-and-quadratic-parts} and
  \eqref{Q-bilinear-def}, respectively, we find that
  \begin{multline}
    \cT _r\bigl[\hat B(r) + e^{-2\gk\tau} \beta(r)\bigr]=
    \label{eq:T-of-B-plus-beta}\\
    e^{-2\gk\tau} \Bigl\{\frac1k r\hat B'(r) - \frac{1}{2(n-1)}
    \bigl(\cL _r[\beta(r)] + 2 \hcQ_r[\hat B(r), \beta(r)]\bigr)\Bigr\}\\
    + e^{-4\gk\tau} \Bigl\{-2\gk\beta(r) + \frac rk \beta'(r) -
    \frac{1}{2(n-1)}\hcQ_r[\beta(r), \beta(r)]\Bigr\}.
  \end{multline}
  Here we use the fact that $\hat B$ satisfies $\cE_r[\hat B]=0$.  For large
  $\tau$, the sign of the \textsc{rhs} is determined by the sign of
  the coefficient of $e^{-2\gk\tau}$.  This coefficient is linear
  inhomogeneous in $\beta$ and its derivatives. To ensure that
  $\hat B(r)+e^{-2\gk\tau}\beta(r)$ is a supersolution for large enough
  $\tau$, we choose $\beta(r)$ in such a way that the factor
  multiplying $e^{-2\gk\tau}$ is positive.  More precisely, we recall
  that $r\hat B'(r)<0$ for all $r>0$, which leads us to choose $\beta$ to
  be a solution of the linear inhomogeneous differential equation
  \begin{equation}
    \cL _r[\beta(r)] + 2 \hcQ_r[\hat B(r), \beta(r)] = 2(n-1)\dA r\hat B'(r),
    \label{eq:beta-diffeq}
  \end{equation}
  with the constant $\dA$ to be chosen below. Applying the definitions of
  $\cL_r$ and $\hcQ_r$, we find that \eqref{eq:beta-diffeq} takes the following form:
  \begin{multline}
    \hat B(r)\frac{\rd^2\beta(r)}{\rd r^2} + \Bigl\{\frac{n-1}{r} - \hat B'(r) -
    \frac{\hat B(r)}{r}\Bigr\}
    \frac{\rd \beta(r)}{\rd r}\\
    + \Bigl\{\hat B''(r) - \frac{\hat B'(r)}{r} + 2(n-1)
    \frac{1-2\hat B(r)}{r^2}\Bigr\}\beta(r) = 2(n-1)\dA r\hat B'(r).
    \label{eq:beta-diffeq-explicit}
  \end{multline}
  This equation has a two-parameter family of solutions. We seek a
  specific solution from this family that leads to the positivity
  of the \textsc{rhs} of \eqref{eq:T-of-B-plus-beta} for
  sufficiently large $\tau$.
  To determine the region on which $\hat B(r)+e^{-2\gk\tau}\beta(r)$ is a
  supersolution, we need to know the asymptotic behavior of our solution
  $\beta(r)$ both as $r\searrow 0$ and as $r\nearrow\infty$.  We now
  show that this asymptotic behavior is the same for all solutions of
  \eqref{eq:beta-diffeq-explicit}, because near both $r=0$ and $r=
  \infty$, we find that the solutions to the homogeneous equation are
  much smaller than any particular solution.

\subsubsection*{Asymptotic behavior at $r=0$:}
At $r=0$, we have $\hat B(0) = 1$, $\hat B'(0) = 0$ and
$\hat B''(0) = -A_4^2\bry''(0)$.  Since $\bry(r)$ is an analytic function
at $r=0$, the \textsc{ode} \eqref{eq:beta-diffeq-explicit} has a regular
singular point at $r=0$, and to leading order can be written as
\[
\frac{\rd^2\beta(r)}{\rd r^2} + \frac{n-2}{r} \frac{\rd \beta(r)}{\rd
  r} - \frac{2(n-1)}{r^2}\beta(r) = -Cr^2
\]
for some constant $C$.  The general solution of this equation is
\[
\beta_0(r)=a_1r^{1-n}+a_2r^2 -\tilde C r^4,
\]
for some $\tilde C$. Discarding the unbounded solution and choosing $a_2=1$,
we conclude that there exists a solution $\beta_p(r)$ of the true
\textsc{ode}~\eqref{eq:beta-diffeq-explicit} that satisfies
$\beta_p(r) = r^2 + o(r^2)$ as $r\searrow0$.


\subsubsection*{Asymptotic behavior at $r=\infty$:}
We have normalized $\bry$ so that $\bry(r)=r^{-2}+\cO (r^{-4})$ as
$r\to\infty$.  This choice yields
$\hat B(r)= \bry(A_4r)=A_4^{-2}r^{-2}+\cO (A_4^{-4} r^{-4})$ near $r=\infty$.
Keeping only the lowest-order terms in $r^{-1}$, we see that for $r$ large,
the \textsc{ode}~\eqref{eq:beta-diffeq-explicit} is a small
perturbation of the \textsc{ode}
\begin{equation}
  \frac{1}{A_4^{2}r^{2}}\beta_\infty''(r)
  +
  \frac{n-1}{r}\beta_\infty'(r)
  +
  \frac{2(n-1)}{r^2}\beta_\infty(r) =
  -\frac{4(n-1)\dA}{A_4^{2}r^{2}},
  \label{TipPert}
\end{equation}
whose general solution is
\begin{equation}
  \beta_\infty(r)=
  \tilde a_1 re^{-\alpha r^2}
  + \tilde a_2 r \int_1^r \rho^{-2}e^{-\hat\alpha(r^2-\rho^2)}\,\rd \rho
  - \frac{2\dA}{A_4^2},
  \label{SolnTipPert}
\end{equation}
where $\hat\alpha:=\frac{n-1}{2}A_4^2$. The first two terms in
\eqref{SolnTipPert}, which solve the homogeneous equation associated
to \eqref{TipPert}, both vanish as $r\nearrow\infty$: the first has
Gaussian decay; the second is $\cO (r^{-2})$. Hence every solution of
the true \textsc{ode}~\eqref{eq:beta-diffeq-explicit}, including the
choice of $\beta_p(r)$ made above, has the property that
$\beta(r)=(1+o(1))\big(-2\dA/A_4^2\big)$ as $r\to\infty$.

To determine how to choose $\dA$, we now estimate
\begin{multline*}
  \cT_r[\hat B(r) + e^{-2\gk\tau}\beta_p(r)]
  =e^{-2\gk\tau} \bigl(\frac1k - \dA\bigr) r\hat B'(r) \\
  +e^{-4\gk\tau} \Bigl( -2\gk\beta_p(r) + \frac{r}{k} \beta_p'(r)
  -\frac{\hcQ_r[\beta_p(r), \beta_p(r)]}{2(n-1)} \Bigr).
\end{multline*}
Since $\hat B'(r)<0$, the first term on the \textsc{rhs} is positive if we
choose $\dA>\frac{1}{k}$. To make a definite choice, we set%
\[
\dA := \frac{1}{k} + 1.
\]
Near $r=0, \infty$, the asymptotics of $\bry$ then imply that
\[
-r\hat B'(r) =
\begin{cases}
  \tilde c_1 r^2 + o(r^2), & (r\searrow0)  \\
  \tilde c_2 r^2 + o(1/r^2) & (r\nearrow\infty)
\end{cases}
\]
for certain positive constants $\tilde c_1$ and $\tilde c_2$, and hence that
\[
-r\hat B'(r) \geq \tilde c \min \bigl\{r^2, r^{-2}\bigr\},
\]
for some constant $\tilde c$. We also know for $\beta_p$ that
\[
\beta_p(r) = \cO(r^2) \quad(r\searrow0),\qquad \beta_p(r) =
-\frac{2\dA}{A_4^2} + \cO(r^{-2}) \quad(r\nearrow\infty).
\]
These expansions can be differentiated.  Hence we have for $0 < r \leq 1$,
\[
\Bigl| -2\gk\beta_p(r) + \frac{r}{k} \beta_p'(r)
-\frac{\hcQ_r[\beta_p(r), \beta_p(r)]}{2(n-1)} \Bigr|\leq Cr^2,
\]
and thus
\begin{align*}
  \cT_r[\hat B(r) + e^{-2\gk\tau}\beta_p(r)] &\geq
  -e^{-2\gk\tau} r\hat B'(r)  - e^{-4\gk\tau} Cr^2\\
  &\geq
  e^{-2\gk\tau}r^2 (\tilde c  - e^{-2\gk\tau} C)\\
  & >0,
\end{align*}
if $\tau$ is large enough.  For $r\ge1$, we have
\[
\Bigl| -2\gk\beta_p(r) + \frac{r}{k} \beta_p'(r)
-\frac{\hcQ_r[\beta_p(r), \beta_p(r)]}{2(n-1)} \Bigr|\leq C,
\]
so that
\begin{align*}
  \cT_r[\hat B(r) + e^{-2\gk\tau}\beta_p(r)] &\geq
  -e^{-2\gk\tau} r\hat B'(r)  - e^{-4\gk\tau} C\\
  &\geq
  e^{-2\gk\tau} (\tilde c r^{-2}  - Ce^{-2\gk\tau} )\\
  & >0,
\end{align*}
provided that $r < \tilde\delta e^{\gk\tau}$, where $\tilde\delta = \sqrt{\tilde c/C}$.

This completes the construction of the supersolution.  A subsolution
can be constructed along the same lines, but one could also just set
$\beta(r)=0$ and observe that \eqref{eq:T-of-B-plus-beta} then takes
the form
\[
\cT _r\bigl[\hat B(r) + e^{-2\gk\tau} \beta(r)\bigr]= \frac1k e^{-2\gk\tau}
r\hat B'(r).
\]
Because $\hat B'(r) < 0$ for all $r$, the \textsc{rhs} of this equation is
negative; hence we see that $\hat B(r) =\bry(A_4r)$ is a subsolution.
\end{proof}

We now use the sub- and supersolutions from Lemma~\ref{lem:tip-barrier} to
construct properly ordered upper and lower barriers for the tip region. To
do this, we specify the constants $A_4^\pm$ as follows, using the small constant
$0<\delta\ll1$ fixed above:
\begin{subequations}
\label{OrderParameters}
\begin{align}
    (A_4^-)^{-2} &= (1+\delta)\left(1+\frac38 B_3^{-2}\right)(1-\delta) c_k,\\
    (A_4^+)^{-2} &= (1-\delta)\left(1+\frac12 B_3^{-2}\right)(1+\delta) c_k.
\end{align}
\end{subequations}
We note several consequences of these choices. First, they ensure that both barriers
are close to the formal solution $\bry(c_k^{-1/2}r)$. Second, they ensure that $A_4^+<A_4^-$;
because $\bry$ is monotone decreasing, this implies that $\bry(A_4^- r)<\bry(A_4^+ r)$
for all $r>0$, which is useful in establishing the desired ordering of the barriers.
Third, the choices above, and in particular the leading $1\pm\delta$ factors, ensure that
$A_4^\pm$ satisfy the strict inequalities in condition~\eqref{eq:A1-A2-gluing-condition}
of Lemma~\ref{PastingLemma} below. We use this fact to guarantee that the barriers constructed
here are compatible with those constructed in Lemma~\ref{OrderIntermediate} within the intersection
of the intermediate and tip regions. These regions overlap for all sufficiently large $\tau$,
depending on $B_2$ in Lemma~\ref{OrderIntermediate} and $B_4$ in Lemma~\ref{lem:tip-barrier}.

\begin{lemma}
\label{OrderTip}
If the parameters $A_4^\pm$ take the values specified by~\eqref{OrderParameters},
and if $\beta$ is the function appearing in Lemma~\ref{lem:tip-barrier}, then
there exists $\tau_5\geq\tau_4$ such that
\begin{equation}
    z_\pm = \bry(A_4^\pm r)\pm e^{-2\gk\tau}\beta(r)
\end{equation}
form a pair of properly ordered barriers that are valid for all points
$0\leq r\leq B_4 e^{\gk\tau}$ and all times $\tau\geq\tau_5$.
\end{lemma}

\begin{proof}
Near $r=0$, one has
$\bry(A_4^\pm r)=1-\mathfrak{b}(A_4^\pm r)^2+\cO(r^4)$ for some constant
$\mathfrak{b}>0$, and $\beta(r)=[1+o(1)]r^2$. Thus as $r\searrow0$, we
see that
\[
    z_+ - z_- =
    \big\{\mathfrak{b}[(A_4^-)^2-(A_4^+)^2]+2[1+o(1)]e^{-\gk\tau}\big\}r^2
    +\cO(r^4)>0.
\]

Near $r=\infty$, one has the expansions
$\bry(A_4^\pm r)=1+(A_4^\pm r)^{-2}+\cO(r^{-4})$
and $\beta=[1+o(1)]\big\{-2(1+1/k)(A_4^\pm)^{-2}\big\}$.
It then follows from the specifications of $A_4^\pm$ given
in \eqref{OrderParameters} that
\[
    z_+ - z_- = \frac{(1-\delta^2)B_3^{-2}c_k}{8}
    \left\{1-2[1+o(1)]\frac{k+1}{k}e^{-2\gk\tau}\right\}+\cO(r^{-4}),
\]
which is positive for all sufficiently large $\tau$ and $r$.

It is straightforward to verify that $z_+ > z_-$ on any bounded interval
$c<r<C$, for all sufficiently large $\tau$. The result follows.
\end{proof}

\subsection{The intermediate-tip interface}
Let $\zi_\pm$ denote the upper and lower barriers constructed in
Lemma~\ref{OrderIntermediate} for the intermediate region,
and let $\zt_\pm$ denote the upper and lower barriers constructed in
Lemma~\ref{OrderTip} for the tip region.  The domains of
definition of these two pairs of functions intersect for sufficiently
large times. To obtain barriers that work throughout the union of
these regions, we must verify that $\zi_\pm$ and $\zt_\pm$ are
properly ordered in their intersection. In what follows, we focus on
upper barriers, omitting the entirely analogous argument for
lower barriers. To avoid notational prolixity, we write $A_1$ and $A_4$ for
$A_1^+$ and $A_4^+$, respectively, in the statement of the lemma and its proof.

\begin{lemma}
  \label{PastingLemma}
  Suppose that $A_1$ and $A_4$ satisfy strict inequalities
  \begin{equation}
    \big(1+\frac38 B_3^{-2}\big)A_1
    < A_4^{-2} <
    \big(1+\frac12 B_3^{-2}\big)A_1,
    \label{eq:A1-A2-gluing-condition}
  \end{equation}
  where $B_3$ is the constant from
  Lemma~\ref{OrderIntermediate}.

  Then there exists a constant $\hat C$ such that if
  \begin{equation}
  \label{ReduceDependence}
  A_3 \geq \hat C
  \left\{A_1^{1/2}+\big(1+\frac12 B_3^{-2}\big)^2 A_1^2\right\},
  \end{equation}
  then one has $A_3 \geq \hat C(A_1^{1/2}+A_4^{-4})$, and consequently
  \begin{subequations}
    \begin{align}
      \zt_+ \leq \zi_+ &\text{ at } r= B_3\sqrt{A_3/A_1}, \\
      \zt_+ \geq \zi_+ &\text{ at } r= 2B_3\sqrt{A_3/A_1},
    \end{align}
    \label{TipIntermediateSubsolutionGluingCondition}
  \end{subequations}
  for all $\tau\geq\tau_6$, where $\tau_6\geq\max\{\tau_3,\tau_5\}$
  is sufficiently large.
\end{lemma}

\begin{proof}
  Near infinity, the Bryant soliton has an expansion of the form (see
  Lemma~18 of \cite{ACK09})
  \[
  \bry(\rho) = \frac{1}{\rho^2} + \frac{\mathfrak b_2}{\rho^4}+
  \frac{\mathfrak b_3}{\rho^6} + \dots
  \]
  Written in terms of the $r$ coordinate, the supersolutions satisfy
  the expansion
  \begin{align*}
    \zt_+ & = \bry (A_4 r) + e^{-2\gk\tau} \beta(r)  \\
    &= A_4^{-2} r^{-2} + {\mathfrak b}_2 A_4^{-4} r^{-4} + \cO(r^{-6})
    + \cO(e^{-2\gk\tau})
  \end{align*}
  as $r\rightarrow\infty$. The $\cO(r^{-6})$ term comes from the asymptotic
  expansion of $\bry(A_2r)$ and is therefore uniform in time (in fact,
  independent of time).

  Using our asymptotic description \eqref{eq:zp-asymptotics} of
  $\zp$, we find that as $u\searrow0$, one has
  \begin{align*}
    \zi_+ & = e^{-2\gk\tau}Z_1(u) + e^{-4\gk\tau}A_3 \zp(u) \\
    & = A_1 e^{-2\gk\tau}u^{-2}(1-u^2)^{1+2\gk} + A_3 e^{-4\gk\tau}
    u^{-4}\bigl(1 + \cO(u^2\log u)\bigr)\\
    & = A_1r^{-2}\bigl(1- e^{-2\gk\tau}r^2\bigr)^{1+2\gk}
    + A_3 r^{-4}\bigl(1 + \cO(u^2\log u)\bigr)\\
    & = A_1 r^{-2} + A_3 r^{-4} + \cO\bigl(\tau e^{-2\gk\tau}\bigr),
  \end{align*}
  where $r=e^{\gk\tau}u$, and where $\cO(\tau e^{-2\gk\tau})$ is
  uniform on any compact $r$ interval. Hence on bounded $r$-intervals,
  one has
  \begin{align*}
    r^2\bigl(\zt_+ - \zi_+\bigr) &= (A_4^{-2} - A_1) +
    \bigl({\mathfrak b_2}A_4^{-4} - A_3\bigr) r^{-2}
    + \cO(r^{-4}) + \cO(\tau e^{-2\gk\tau})\\
    &= (A_4^{-2} - A_1) + \bigl({\mathfrak b_2}A_4^{-4}+ \cO(r^{-2}) -
    A_3\bigr) r^{-2} + \cO(\tau e^{-2\gk\tau}).
  \end{align*}

  We want this quantity to change from negative to positive as $r$
  increases from $B_3\sqrt{A_3/A_1}$ to $2B_3\sqrt{A_3/A_1}$. At
  $r=B_3\sqrt{A_3/A_1}$, we have
  \[
  r^2\bigl(\zt_+ - \zi_+\bigr) = A_4^{-2} - A_1 + \left\{
    \frac{{\mathfrak b_2}A_1}{A_3 A_4^{4}} +
    \cO\Big(\bigl(\frac{A_1}{A_3}\bigr)^2\Big)-A_1 \right\} B_3^{-2} +
  \cO(\tau e^{-2\gk\tau}),
  \]
  and at $r=2B_3\sqrt{A_3/A_1}$, we have
  \[
  r^2\bigl(\zt_+ - \zi_+\bigr) = A_4^{-2} - A_1 + \frac14 \left\{
    \frac{{\mathfrak b_2}A_1}{A_3 A_4^{4}} +
    \cO\Big(\bigl(\frac{A_1}{A_3}\bigr)^2\Big)-A_1 \right\} B_3^{-2} +
  \cO(\tau e^{-2\gk\tau}).
  \]
  If $A_3 \geq \hat C A_4^{-4}$ and $A_3 \geq \hat C \sqrt{A_1}$ both
  hold for sufficiently large $\hat C$, then
  \[
  \left| \frac{{\mathfrak b_2}A_1}{A_3 A_4^{4}} +
    \cO\bigl(\frac{A_1^2}{A_3^2}\bigr)\right| \leq \frac{A_1}{2},
  \]
  and we find that
  \begin{align*}
    r^2\bigl(\zt_+ - \zi_+\bigr) &\leq A_4^{-2}
    - \big(1+\frac12 B_3^{-2}\big) A_1 + \cO(\tau e^{-2\gk\tau})
    \text{ at }r = B_3\sqrt{{A_3}/{A_1}}, \\
    r^2\bigl(\zt_+ - \zi_+\bigr) &\geq A_4^{-2}
    - \big(1+\frac38 B_3^{-2}\big)A_1 + \cO(\tau e^{-2\gk\tau})
    \text{ at }r = 2B_3\sqrt{{A_3}/{A_1}}.
  \end{align*}
  From this it is clear that if the strict inequalities
  \eqref{eq:A1-A2-gluing-condition} hold, then the upper barriers
  $\zt_+$ and $\zi_+$ satisfy the patching condition
  \eqref{TipIntermediateSubsolutionGluingCondition} for all
  sufficiently large $\tau$.
\end{proof}

\subsection{Barriers in the outer region}
\label{OuterBanks}
If $k\geq4$ is even, then as explained in Section~\ref{OuterTube},
solutions in $\Xi_\ve$ are controlled by the barriers constructed
above.  If $k\geq3$ is odd, we need additional barriers to define
$(\partial\Xi_\ve)_+$ in a suitable subset of $u\in(1,e^{\tau/2})$.

\begin{lemma}
  \label{OuterBarrierLemma}
  For any $A_5>0$, there exist a function
  $\tilde\zeta:(1,\infty)\to\R$ and constants $A_6$ and $B_5>0$
  depending only on $A_5$ such that
  \begin{equation}
    \tz_\pm=A_5 e^{-2\gk\tau}(u^2-1)^{1+2\gk}u^{-2}
    \pm A_6 e^{-4\gk\tau}\tzz(u)
  \end{equation} are sub- $(\tz_-)$ and super- $(\tz_+)$ solutions in the region
  \begin{equation}
    1+B_5 e^{-\mu_k\tau}\leq u \leq B_5^{-1} e^{\tau/2}.
  \end{equation}
\end{lemma}

\begin{proof}
  We sketch the proof, because it is very similar to that of
  Lemma~\ref{lem:intermediate-barrier}.

  Let $\tz_1(u)=(u^2-1)^{1+2\gk}u^{-2}$.  Then for $1\leq u<\infty$,
  one estimates $|\tz_1|\leq C(u-1)^{2\gk+1}u^{2\gk-1}$, $|\tz_1'|\leq
  C(u-1)^{2\gk}u^{2\gk-1}$, and $|\tz_1''|\leq
  C(u-1)^{2\gk-1}u^{2\gk-1}$. Hence
  \[
  |\cQ _u[\tz_1]|\leq C (u-1)^{4\gk}u^{4\gk-2}.
  \]
  This motivates us to choose the function $\tzz$ to be a solution of
  the \textsc{ode}
  \[
  \frac12(u-u^{-1})\frac{\rd\tzz}{\rd
    u}-(4\gk+u^{-2})\tzz=(u-1)^{4\gk}u^{4\gk-2}.
  \]
  Its solutions are $\tzz=\tzz_p+Z_{2,{\rm hom}}$, where
  $\tzz_p$ is any particular solution, and $Z_{2,{\rm hom}}$ solves
  the associated homogeneous \textsc{ode}. We may thus choose a
  solution $\tzz$ such that $\tzz\sim(u-1)^{4\gk}$ (the behavior of
  $\tzz_p$) as $u\searrow1$ and $\tzz\sim u^{8\gk}$ (the behavior of
  $Z_{2,{\rm hom}}$) as $u\to\infty$.  Estimating that $|\tzz|\leq
  C(u-1)^{4\gk}u^{4\gk}$, $|\tzz'|\leq C(u-1)^{4\gk-1}u^{4\gk}$, and
  $|\tzz''|\leq C(u-1)^{4\gk-2}u^{4\gk}$, we observe that
  \[
  |\cQ _u[\tzz]|\leq C(u-1)^{8\gk-2}u^{8\gk},
  \]
  and
  \[
  |\hcQ_u {[\tz_1,\tzz]}|\leq C(u-1)^{6\gk-1}u^{6\gk-1}.
  \]

  With these estimates for $\tz_1(u)$ established, we write
  \[
  \tz_\pm=A_5 e^{-2\gk\tau}\tz_1(u) + A_6 e^{-4\gk\tau}\tzz(u),
  \]
  as above, and seek to determine $A_6$ (depending only on $A_5$) so that
  $\tz_+$ is a supersolution. Based on the estimates above, we have
  \[
  e^{4\gk\tau} \cD _u[\tz_+] \geq
  \big\{A_6-\frac{C}{2(n-1)}A_5^2\big\}(u-1)^{4\gk}u^{4\gk-2}
  +\cO(e^{-2\gk\tau}).
  \]
  This leads us to choose $A_6:=1+\frac{C}{2(n-1)}A_5^2$.  Then one
  obtains
  \[
  e^{4\gk\tau} \cD _u[\tz_+] \geq
  (u-1)^{4\gk}u^{4\gk-2}\big\{1-CA_5A_6X-C(A_6X)^2\big\},
  \]
  where $X:=e^{-2\gk\tau}(u-1)^{2\gk-1}u^{2\gk+1}$.  It follows that
  $\tz_+$ is a subsolution in any spacetime region where
  $CA_5(A_6X)+C(A_6X)^2<1$.

  For $u>1$ near $1$, we observe that there is a constant $B'$
  depending only $A_5$ such that $\tz_+$ is a supersolution if
  $(u-1)^{1-2\gk}\geq B' e^{-2\gk\tau}$, hence if $v^{2/k}\geq
  B'e^{-2\gk\tau}$.  It is easy to see that this is implied by the
  constraint $u\geq1+B_5e^{-\mu_k\tau}$ if $B_5$ is large enough.

  For $u\gg1$, we observe that there is a constant $B''$ depending
  only on $A_5$ such that $\tz_+$ is a supersolution if $B'' u\leq
  e^{\tau/2}$, hence if $B'' \psi\leq1$.

  Similar considerations show that $\tz_-$ is a subsolution.
\end{proof}

As in the other regions, by making suitable modifications to the constants,
one can ensure that the sub- and supersolutions of Lemma~\ref{OuterBarrierLemma}
are properly ordered upper and lower barriers in the outer region. To do so, we
specify constants $A_5^\pm=(1\pm\delta)c_k$, where $c_k$ is defined in \eqref{eq:c1-bk-relation},
and $0<\delta\ll1$ is the small constant fixed above.
We define $A_7:=\max\{A_6(A_5^-),A_6(A_5^+)\}$, and we set $B_6:=\max\{B_5(A_5^-),B_5(A_5^+)\}$.
Then we have the following:

\begin{lemma}
\label{OrderOuter}
There exist $B_7\geq B_6$ depending only on $b_k$ such that
if $\tzz$ is the function appearing in Lemma~\ref{OuterBarrierLemma}, then
\begin{equation}
    \tz_\pm=(1\pm\delta)c_k e^{-2\gk\tau}(u^2-1)^{1+2\gk}u^{-2}
    \pm A_7 e^{-4\gk\tau}\tzz(u)
\end{equation}
are a pair of properly ordered barriers in the region
\begin{equation}
    1+B_5 e^{-\mu_k\tau}\leq u \leq (B_7)^{-1} e^{\tau/2}.
\end{equation}
\end{lemma}

\begin{proof}
We sketch the proof, which is similar to that of Lemma~\ref{OrderIntermediate}, but somewhat simpler.

As $u\searrow1$, one calculates that $\tzz\sim(u-1)^{4\gk}>0$. And as $u\rightarrow\infty$,
one has $e^{-2\gk\tau}(u^2-1)^{1+2\gk}u^{-2}\sim e^{-2\gk\tau}u^{4\gk}$
and $e^{-4\gk\tau}\tzz\sim \big(e^{-2\gk\tau}u^{4\gk}\big)^2$.
It follows that the barriers are properly ordered provided that
$B_7 e^{-\tau/2}u\leq1$ for some $B_7\geq B_6$ sufficiently large.
\end{proof}

\section{Analysis of the parabolic region}\label{ExitStrategy}

Ricci flow solutions cannot escape from the portion of
$\partial\Xi_{\ve}$ that is associated with the tip, intermediate,
and outer regions, and is explicitly constructed in
Section~\ref{PuttingUpBarriers} using upper and lower barriers.
Solutions can, however, escape from the portion of $\partial\Xi_{\ve}$
that is associated with the parabolic region, and is defined by the
inequalities listed in \eqref{vkminusL2}--\eqref{DefineParabolicPW}.
In this section, after establishing a collection of key estimates in
Lemmas~\ref{FirstDerivativeBound}--\ref{StrapYourBoots}, we derive the
exit and entrapment results, Lemmas~\ref{ExitLemma}--\ref{PWentrapment},
that relate to these inequalities and control $\Xi_\ve$ in the parabolic
region. As noted in Section~\ref{Outline}, these results play a crucial
role in our proof of our main theorem.

\subsection{The parabolic-intermediate interface}
We recall from equation~\eqref{eq:z-expansion-near-neck} that in the
parabolic region, the formal solution indexed by $n$, $k$, and $b_k$ satisfies
\[
z\eh 2(n-1)k^2(-b_k)^{2/k}e^{-2\gk\tau}(1-u)^{1+2\gk}
\]
for $|\sigma|\gg1$; here we use the identity $2-2/k=1+2\gk$. As
shown in Lemmas~\ref{lem:intermediate-barrier}--\ref{PastingLemma}, we
can construct properly ordered and patched barriers that encase this
formal solution and are valid for $\tau\geq\tau_6$ and
$u\leq1-\hb e^{-\mu_k\tau}$; i.e., for $e^{\mu_k\tau}|v|\geq\hb$.
Here $\tau_6$ and $\hb:=B_3(2A_3/c_k)^{k/2}$ depend only on $b_k$.
These barriers define $\Xi_\ve$ in the intermediate and tip regions.
We now use them to derive information about the parabolic-intermediate
interface for solutions belonging to $\Xi_\ve$. This interface corresponds
to $|\sigma|$ sufficiently large, where ``sufficiently large'' is made precise
below.

\begin{lemma}
  \label{FirstDerivativeBound}
  There exist constants $A_8^\pm$, $\hb$ and $\tau_7\geq\tau_6$
  depending only on $b_k$ such that every solution in $\Xi_\ve$
  satisfies
  \begin{equation}
    \label{derivbd}
    \frac{\sqrt{A_8^-}}{k}
    \leq\pl_\sigma\big\{(-e^{\mu_k\tau}v)^{\frac1k}\big\}
    \leq\frac{\sqrt{A_8^+}}{k}
  \end{equation}
  for $\hb\leq e^{\mu_k\tau}|v|\leq\frac12 e^{\mu_k\tau}$ and
  $\tau\geq\tau_7$.
\end{lemma}

\begin{proof}
  By using \eqref{eq:zp-asymptotics} to bound
  $|\zeta|\leq C(1-u)^{4\gk}$  for
  $\frac12\leq u\leq1$, we see from
  Lemmas~\ref{lem:intermediate-barrier}--\ref{OrderIntermediate} and the fact that
  $2\gk=\mu_k(1-2\gk)$ that there exist constants $\ham,\hap$ such that any solution
  in $\Xi_\ve$ satisfies
  \begin{equation*}
    \ham e^{-2\gk\tau}u^{-2} (1-u^2)^{1+2\gk}
    \leq z\leq
    \hap e^{-2\gk\tau}u^{-2} (1-u^2)^{1+2\gk}
  \end{equation*}
  for $\frac12\leq u\leq 1-\hb e^{-\mu_k\tau}$ and
  $\tau\geq\tau_7\geq\tau_6$, where $\hapm$, $\hb$, and
  $\tau_7$ depend only on $b_k$. This implies that there exist
  constants $A_8^\pm$ depending only on $b_k$ such that one can
  bound $\big(\pl_\sigma(v^{\frac 1k})\big)^2=k^{-2}(1-u)^{\frac
    2k-2}z$ in the same space-time region as follows:
  \[
  \frac{A_8^-}{k^2}e^{-2\gk\tau} \leq\big(\pl_\sigma(v^{\frac
    1k})\big)^2 \leq\frac{A_8^+}{k^2}e^{-2\gk\tau}.
  \]
  Recalling that $-v=1-u>0$ and $\mu_k/k=\gk$, we obtain \eqref{derivbd}.
\end{proof}

We now show that Lemma~\ref{FirstDerivativeBound} gives pointwise
control of $e^{\mu_k\tau}|v|$ for large $|\sigma|$, specifically in
the parabolic-intermediate intersection. Below, we use properties of
the operator $\cA $ to provide pointwise control of this quantity for
smaller $|\sigma|$.  Both of these arguments are facilitated by a
judicious choice of the parameter $P=2(B/|b_k|)^{1/k}$ defined in
equation~\eqref{DefineP}.  Noting that there is a universal constant
$C_k$ with $|h_k(\sigma)-\sigma^k|\leq C_k|\sigma|^{k-2}$ for
$|\sigma|\geq1$, we choose $B\geq\hb$ large enough, depending only on
$b_k$, so that
\begin{equation}
  P\geq\frac{10}{3}\sqrt{C_k}.
  \label{ChooseP}
\end{equation}

\begin{lemma}
  \label{BrandenburgerTor}
  If $P=2(B/|b_k|)^{1/k}$ is chosen as in \eqref{ChooseP}, depending only on $b_k$, then
  for all sufficiently small $\ve$, depending on $\params$, every solution in $\Xi_\ve$ satisfies
  \[
  e^{\mu_k\tau}\big|v\big(\frac25 P,\tau\big)\big|<B,
  \qquad\text{and}\qquad
  e^{\mu_k\tau}\big|v\big(\sigma,\tau\big)\big|>B
  \quad\text{if}\quad\frac35 P\leq|\sigma|\leq P.
  \]
\end{lemma}

\begin{proof}
  We first show that the bounds are true for the formal solution.  Our
  choice of $P$ ensures that
  \begin{equation}
    |b_k|\,h_k\big(\frac25 P\big)
    \leq\left(\frac45\right)^k B\left\{1+C_k\left(\frac54\right)^2\left(\frac 2P\right)^2\right\}
    \leq\left(\frac45\right)^{k-2}B \leq\frac45 B.
    \label{BelowB}
  \end{equation}
  On the other hand, if $|\sigma|\geq\frac35 P$, then
  \[
  |b_k|h_k(\sigma)
  \geq|b_k||\sigma|^k\left\{1-C_k|\sigma|^{-2}\right\}
  \geq\frac34|b_k||\sigma|^k.
  \]
  In particular,
  \[
  |b_k|\,h_k\big(\frac35 P\big) >\left(\frac65\right)^{k-2}B
  \geq\frac65 B.
  \]

  Now because $\|\cdot\|_{L^2(\cP ; \Dh)}\leq\|\cdot\nh$, one may
  apply a Gagliardo--Nirenberg interpolation inequality (e.g., see
  p.~125 of \cite{Nirenberg59}) in the compact parabolic region $\cP $
  to see that
  \begin{multline*}
    \|e^{\mu_k\tau}v-b_k h_k\|_{L^\infty(\cP )}\\
    \leq C\left\{ \|\pl_\sigma(e^{\mu_k\tau}v-b_k h_k)\|_{L^\infty(\cP
        )}^{1/3} \|e^{\mu_k\tau}v-b_k
      h_k\nh^{2/3}+\|e^{\mu_k\tau}v-b_k h_k\nh \right\}.
  \end{multline*}
  Conditions~\eqref{vkminusL2}--\eqref{vkplusL2} imply that
  $\|e^{\mu_k\tau}v-b_k h_k\nh\leq3\ve$ for solutions in $\Xi_\ve$.
  Hence by condition~\eqref{DefineParabolicPW}, one has
  \[
  \|e^{\mu_k\tau}v-b_k h_k\|_{L^\infty(\cP )}\leq
  C(1+W^{1/3})\ve^{2/3}.
  \]
  So the result holds for all sufficiently small $\ve$, depending on $\params$,
  where the constant $W$ is chosen in Lemma~\ref{PWentrapment} below.
\end{proof}

\begin{remark}
  Because $v$ is continuous, it follows that there exists
  $\sigma_B(\tau)\in\left(\frac25 P,\frac35 P\right)$ such that
  $e^{\mu_k\tau}v(\sigma_B(\tau),\tau)=-B$.
\end{remark}

Hereafter, we assume that $\ve=\ve(b_k,W)$ is sufficiently small, as
indicated in Lemma~\ref{BrandenburgerTor}. One then has the following
result, anticipated in Section~\ref{Outline}:

\begin{lemma}
  \label{k-bound}
  There exist $A^\pm=A_8^\pm$, $\delta_1$, and $\tau_8\geq\tau_7$ depending
  only on $b_k$ such that for $\tau\geq\tau_8$, every solution in $\Xi_\ve$ satisfies
  \[
  \frac{\sqrt{A^-}}{k}
  \leq\pl_\sigma\big\{(-e^{\mu_k\tau}v)^{\frac1k}\big\}
  \leq\frac{\sqrt{A^+}}{k}
  \]
  in the parabolic-intermediate interface $|\sigma|\in\left(\frac35
    P,\delta_1e^{\gk\tau}\right)$, where $P=2(B/|b_k|)^{1/k}$.
\end{lemma}

\begin{proof}
  Lemma~\ref{BrandenburgerTor} implies that any solution $v\in\Xi_\ve$
  obeys the derivative estimate of Lemma~\ref{FirstDerivativeBound}
  for all $|\sigma|\geq\frac35P$ such that $|v|\leq\frac12$. Further,
  so long as Lemma~\ref{FirstDerivativeBound} applies, one has
  \[
  (e^{\mu_k\tau}|v|)^{\frac{1}{k}}\leq B^{\frac{1}{k}}
  +\frac{\sqrt{A^+}}{k}\left(|\sigma|-\frac35P\right).
  \]
  The \textsc{rhs} is bounded above by $\big(\frac12
  e^{\mu_k\tau}\big)^{\frac{1}{k}}$ provided that
  $|\sigma|\leq\delta_1 e^{\gk\tau}$ and $\tau\geq\tau_8$, where we
  choose $\tau_8\geq\tau_7$ large enough so that
  $\delta_1 e^{\gk\tau_8}\geq2P$.
\end{proof}

We next establish the following bound, also anticipated in
Section~\ref{Outline}:

\begin{lemma}
  \label{GrowthCondition}
  There exist $C_\pm$ depending on $\params$ such that for
  $|\sigma|\in\left(\frac35 P, \delta_1e^{\gk\tau}\right)$
  and $\tau\geq\tau_8$, every solution in $\Xi_\ve$ satisfies
  \[
  C_-|\sigma|^k \leq e^{\mu_k\tau}|v| \leq C_+|\sigma|^k.
  \]
\end{lemma}

\begin{proof}
  With $C_-:=\big(\min\{B^{1/k}/(\frac35 P),\sqrt{A_8^-}/k\}\big)^k$, we
  apply Lemmas~\ref{BrandenburgerTor}--\ref{k-bound} to see
  that $(e^{\mu_k\tau}|v|)^{1/k}\geq\int_0^\sigma
  C_-^{1/k}\,\mathrm{d}\sigma$ for $\frac35P\leq\sigma\leq\delta_1
  e^{\gk\tau_1}$. An analogous estimate holds for $-\delta_1
  e^{\gk\tau_1}\leq\sigma\leq-\frac35P$ in the case that $k$ is even.
  This proves the lower bound.

  For the upper bound, note that $(e^{\mu_k\tau}\big|v\big(\frac25
  P,\tau\big)\big|)^{1/k}\leq B^{1/k}$ holds by
  Lemma~\ref{BrandenburgerTor}. In the region $\frac25
  P\leq|\sigma|\leq\frac35 P$, one has
  $|b_k|h_k(\sigma)\geq\frac{7}{16}\big(\frac45\big)^kB$ for the
  formal solution, and hence
  $e^{\mu_k\tau}|v|\geq\big(\frac45\big)^{k+3}B$, as in the proof of
  Lemma~\ref{BrandenburgerTor}. So for $\frac25
  P\leq|\sigma|\leq\frac35 P$, condition~\eqref{DefineParabolicPW}
  implies that
  $\pl_\sigma\left(e^{\mu_k\tau}|v|\right)^{\frac1k}\leq\vartheta(b_k)$,
  where
  \[
  \vartheta(b_k):=\frac{W}{k}\left(\big(\frac45\big)^{k+3}B\right)^{\frac1k-1}.
  \]
  If we now choose
  $C_+:=\big(\max\{B^{1/k}/(\frac25 P),\vartheta(b_k),\sqrt{A^+}/k\}\big)^k$,
  then the upper bound follows.
\end{proof}

\begin{remark}
  Combining Lemma~\ref{k-bound} with Lemma~\ref{GrowthCondition}, one
  obtains
  \begin{equation}
    C_-'|\sigma|^{k-1}\leq e^{\mu_k\tau}|v_\sigma|\leq C_+'|\sigma|^{k-1}
    \label{DifferentiateGrowthCondition}
  \end{equation}
  for $|\sigma|\in\big(\frac35 P,\delta_1 e^{\gk\tau}\big)$, where
  $C_\pm' := C_{\pm}^{1-1/k}\sqrt{A^\pm}$ depend on $\params$.
\end{remark}

\subsection{Estimating the ``error terms''}
The results in the previous section prepare us to derive useful
bounds for $|(\pl_\tau-\cA )v|$ and $\|(\pl_\tau-\cA )\vh\nh$.

\begin{lemma}
  \label{PointwiseBoundsinP}
  There exist $C_1$ depending only on $b_k$, and $\tau_9\geq\tau_8$ depending
  on $\params$, such that for every solution in $\Xi_\ve$, the pointwise bound
  \[
  |(\pl_\tau-\cA )v|\leq C_1 W^2 e^{-2\mu_k\tau}
  \]
  holds in the parabolic region $\cP$ for all $\tau\geq\tau_9$.
\end{lemma}

\begin{proof}
  Recalling that $\vh=v$ in the parabolic region, we determine from
  \eqref{CutoffEvolution} that
  $(\pl_\tau-\cA )v=N_{\mathrm{loc}}-nv_\sigma I$ in $\cP $, where
  \[
  N_{\mathrm{loc}}=\frac{2(n-1)v_{\sigma}^2-v^2}{2(1+v)}
  \]
  and
  \[
  I =\int_0^{\sigma}
  \frac{v_{\hat{\sigma}\hat{\sigma}}(\hat{\sigma},\tau)}{1+v(\hat{\sigma},\tau)}
  \, \rd \hat{\sigma}.
  \]

  If we choose $\tau_9\geq\tau_8$, depending on $\params$, so that
  $We^{-\mu_k\tau_9}\leq\frac12$, then it follows from
  \eqref{DefineParabolicPW} that one has $1+v\geq
  1-We^{-\mu_k\tau}\geq \frac12$ for $|\sigma|\leq P$ and
  $\tau\geq\tau_9$.  Inequality \eqref{DefineParabolicPW} also implies
  that there exists $C_0$ (depending only on $b_k$) such that
  $|N_{\mathrm{loc}}|\leq C_0 W^2 e^{-2\mu_k\tau}$ holds in $\cP $ for
  $\tau\geq\tau_9$.

  Integration by parts shows that
  \begin{equation}
    \label{intparts}
    I=\frac{v_\sigma}{1+v}\Big|_0^\sigma
    +\int_0^\sigma \frac{v_{\hat\sigma}^2}{1+v}\,\mathrm{d}\hat\sigma.
  \end{equation}
  In $\cP $, the first term on the \textsc{rhs} of \eqref{intparts} is
  bounded by $4We^{-\mu_k\tau}$, while the second is bounded by
  $2PW^2e^{-2\mu_k\tau}\leq4PWe^{-\mu_k\tau}$. Noting that
  $|v_\sigma|\leq W e^{-\mu_k\tau}$ in $\cP $, we obtain a suitable
  bound for $|v_\sigma I|$ and thus complete the proof.
\end{proof}

\begin{lemma}
  \label{ErrorSmall}
  There exist $C_2$ depending only on $b_k$, and $\tau_{10}\geq\tau_9$
  depending on $\params$, such that for all $\tau\geq\tau_{10}$,
  every solution in $\Xi_\ve$ satisfies
  \[
  \|(\pl_\tau-\cA )\vh\nh \leq C_2 e^{-\frac32\mu_k\tau}.
  \]
\end{lemma}

\begin{proof}
It follows from Lemma~\ref{PointwiseBoundsinP} that there exists $\hat C_0$
depending only on $b_k$ such that a pointwise bound
$|(\pl_\tau-\cA )v|\leq \hat C_0 e^{-\frac32\mu_k\tau}$
holds for $|\sigma|\leq P$, and for all sufficiently large $\tau$, depending
on $\params$. Thus to complete the proof, we concentrate on the rest of the
support of $(\pl_\tau-\cA )v$, namely $P\leq|\sigma|\leq\frac65 e^{\gk\tau/5}$.
We assume that $\tau$ is large enough so that $e^{\gk\tau/5}\leq\delta_1 e^{\gk\tau}$.

It follows from equation~\eqref{CutoffEvolution} that
$(\pl_\tau-\cA)\vh=\eta(N_{\mathrm{loc}}-n v_\sigma I)+E$ for
$|\sigma|\in\big(P,\frac65 e^{\gk\tau/5}\big)$, where
$E=\left(\eta_\tau-\eta_{\sigma\sigma}+\frac{\sigma}{2}\eta_\sigma\right)v-2\eta_\sigma v_\sigma$.

For $\ve$ chosen as in Lemma~\ref{BrandenburgerTor}, it follows from Lemma~\ref{k-bound}
that $1+v\geq\frac12$ in the interval $\big(\frac35P,\delta_1e^{\gk\tau}\big)$.  In the
same region, one has $C_-|\sigma|^k \leq e^{\mu_k\tau}|v| \leq C_+|\sigma|^k$ and
$C_-'|\sigma|^{k-1}\leq e^{\mu_k\tau}|v_\sigma|\leq C_+'|\sigma|^{k-1}$ as consequences
of Lemma~\ref{GrowthCondition} and estimate~\eqref{DifferentiateGrowthCondition}, respectively.

Combining these inequalities, and again integrating by parts to evaluate $I$, we obtain
$\hat C_1$ depending on $\params$ and $\hat C_2$ depending only on $b_k$ such that one has
\[
  |N_{\mathrm{loc}}|+|v_\sigma I| \leq \hat C_1 e^{-2\mu_k\tau}|\sigma|^{2k}
  \leq \hat C_2 e^{-\frac32\mu_k\tau}
\]
in the interval $P\leq|\sigma|\leq\frac65 e^{\gk\tau/5}$, at all $\tau$ sufficiently
large, depending on $\params$. Here we use the fact that $2k\gk/5<\mu_k/2$.

Similarly, the pointwise estimates for $v$ and $v_\sigma$ above further imply that there
exist constants $\hat C_3,\hat C_4$ depending on $\params$, and $\hat C_5$ depending only
on $b_k$, such that the estimates
  \begin{align*}
    \|E\nh^2 &\leq \hat C_3 e^{-2\mu_k\tau}
    \int_{e^{\gk\tau/5}}^{\infty}|\sigma|^{2k+2}\,e^{-\sigma^2/4}\,\rd\sigma\\
    &\leq \hat C_4 e^{-2\mu_k\tau}
        \exp{\left(-\frac{e^{2\gk\tau/5}}{4}\right)} e^{(2k+1)\gk\tau/5}\\
    &\leq \hat C_5 e^{-3\mu_k\tau}
  \end{align*}
hold for all $\tau$ sufficiently large, depending on $\params$. The result follows.
\end{proof}

Finally, we derive pointwise bounds for $\vh$ that are independent of $W$ ---
bounds that apply to solutions originating from initial data we construct
in Section~\ref{InitialData} below.

Given a smooth function $f(\sigma)$ and a constant $R>0$, we define
\[
  \|f\|_{C^1(R)} := \sup_{|\sigma|\leq R}\big(|f(\sigma)| + |f_\sigma(\sigma)|\big).
\]

\begin{lemma}
\label{StrapYourBoots}
    If at time $\tau_{11} \geq \tau_{10}$, one has
  \[
    \|\vh(\cdot,\tau_{11})\|_{C^1(2P)} \leq M e^{-\mu_k\tau_{11}}
  \]
  for some constant $M>0$, then for all $\tau\geq\tau_{11}$, one has
  \begin{equation}
    \|v(\cdot,\tau)\|_{C^1(P)}\leq C_6 (1+M) e^{-\mu_k\tau},
    \label{eq:v-pointwise-estimate}
  \end{equation}
  where $C_6\geq1$ depends only on $b_k$ (and is in particular independent of $W$).
\end{lemma}

\begin{proof}
By \eqref{CutoffEvolution}, we have $\vh_\tau - \cA\vh = F(\sigma, \tau)$,
where $F=\eta N[v]+E[\eta,v]$. Thus for $\tau\geq \tau_{11}+1$, the variation
of constants formula lets us write
\begin{equation}
  \vh(\cdot, \tau) =
  e^{\cA}\vh(\cdot, \tau-1) +
  \int_{\tau-1}^\tau e^{(\tau-\tau')\cA} F(\cdot, \tau') \, \rd \tau'.
  \label{VariationOfConstants}
\end{equation}
Standard regularizing estimates for the operator $\cA$ imply that
\[
  \|e^{\cA}\vh(\cdot, \tau-1) \|_{C^1(P)}
  \leq
  \check C_1 \|\vh(\cdot, \tau-1)\|_\hs,
\]
and for $0<\tau-\tau'<1$, that
\[
  \|e^{(\tau-\tau')\cA} F(\cdot, \tau')\|_{C^1(P)}
  \leq
  \check C_1(\tau-\tau')^{-3/4} \|F(\cdot, \tau')\|_\hs,
\]
where $\check C_1$ depends only on $P$, hence by \eqref{DefineP},
only on $b_k$. Conditions~\eqref{vkminusL2}--\eqref{vkplusL2} imply that
there exists $\check C_2$ depending only on $b_k$ such that
\[
    \|\vh(\cdot,\tau-1)\nh\leq \check C_2 e^{-\mu_k\tau}.
\]
Furthermore, because $\tau_{11}\geq\tau_{10}=\tau_{10}(b_k,W)$,
Lemma~\ref{ErrorSmall} provides $C_2$ depending only on $b_k$ such that
\[
  \|F(\cdot,\tau')\|_\hs \leq C_2 e^{-\frac32\mu_k\tau'}.
\]
Combining these estimates and integrating \eqref{VariationOfConstants}
gives \eqref{eq:v-pointwise-estimate} for times $\tau\geq \tau_{11}+1$.

To prove \eqref{eq:v-pointwise-estimate} for $\tau \in (\tau_{11}, \tau_{11}+1)$,
we again use variation of constants, writing
\begin{equation}
  \vh(\cdot, \tau) =
  e^{(\tau-\tau_{11})\cA}\vh(\cdot, \tau_{11}) +
  \int_{\tau_{11}}^\tau e^{(\tau-\tau')\cA} F(\cdot, \tau') \, \rd \tau'.
  \label{eq:VoC-short-time}
\end{equation}
As above, we can use the estimate for $\|F(\cdot,\tau')\nh$ given by Lemma~\ref{ErrorSmall}
to obtain a satisfactory $C^1$ estimate for the integral term. However, we cannot
use the smoothing properties of the operator $e^{\theta\cA}$ to estimate the
first term on the \textsc{rhs}, since the time delay $\theta = \tau-\tau_{11}$
may now be arbitrarily short. But since $\cA$ is a nondegenerate parabolic operator,
there exists $\check C_3$ depending only on $P$, hence only on $b_k$, such that for
$\theta\in(0,1)$, one has
\[
  \|e^{\theta\cA} \vh(\cdot, \tau_{11}) \|_{C^1(P)} \leq
  \check C_3 \| \vh(\cdot, \tau_{11}) \|_{C^1(2P)}.
\]
Since $\|\vh(\cdot,\tau_{11})\|_{C^1(2P)} \leq M e^{-\mu_k\tau_{11}}$ by hypothesis,
this lets us estimate the first term on the \textsc{rhs} of \eqref{eq:VoC-short-time}.
Thus we obtain \eqref{eq:v-pointwise-estimate} for $\tau_{11}<\tau<\tau_{11}+1$,
which completes the proof.
\end{proof}

\subsection{Exit and entrapment results}\label{InOrOut}

We now use the estimates derived above to prove the exit and
entrapment results corresponding to various portions of the boundary
of $\Xi_{\ve}$ in the parabolic region. Our first result shows
that solutions immediately exit if they contact $(\pl\Xi_\ve)_-$.

\begin{lemma}
  \label{ExitLemma}
  There exists $\bt\geq\tau_{10}$ depending on $\params$ such that if
  a solution $v\in\Xi_\ve$ contacts $\pl\Xi_\ve$ by achieving equality
  in \eqref{vkminusL2} at $\tau\geq\bt$, then it immediately exits $\Xi_\ve$.
\end{lemma}

\begin{proof}
  We recall from Section~\ref{Structure} that the projection $v_{k-}$
  represents the rapidly-growing perturbations of a formal
  solution. Since the projection from $\vh$ to $v_{k-}$ commutes with
  the operator $\cA$, and since the projections $\{v_{k-}, v_k, v_{k+}\}$
  are pairwise orthogonal, one has
  \[
  \frac12\frac{\rd}{\rd\tau}\|v_{k-}\nh^2 =\lh
  v_{k-},\dfdtau\vh\sigma\rh =\lh v_{k-},\cA \vh+(\pl_\tau-\cA
  )\vh\rh.
  \]
  We readily verify the inequality $\lh
  v_{k-},\cA\vh\rh\geq-\mu_{k-1}\|v_{k-}\nh^2$, from which it follows
  that
  \[
  \frac12\frac{\rd}{\rd\tau}\|e^{\mu_k\tau}v_{k-}\nh^2
  \geq(\mu_k-\mu_{k-1})\|e^{\mu_k\tau}v_{k-}\nh^2 -|\lh
  e^{\mu_k\tau}v_{k-},e^{\mu_k\tau}(\pl_\tau-\cA )\vh\rh|.
  \]
  Hence using Lemma~\ref{ErrorSmall} and Cauchy--Schwarz, we obtain
  \[
  \frac12\frac{\rd}{\rd\tau}\|e^{\mu_k\tau}v_{k-}\nh^2\Big|_{\|e^{\mu_k\tau}v_{k-}\nh=\ve}
  \geq\ve\big(\frac{\ve}{2}-C_2 e^{-\frac12\mu_k\tau}\big)>0
  \]
  at all times $\tau\geq\bt$, for $\bt\geq\tau_{10}$ chosen sufficiently
  large, depending on $\params$.
\end{proof}

Our next three results are entrapment lemmas.

\begin{lemma}
  \label{FirstEntrapmentLemma}
  There exists $\bt\geq\tau_{10}$ depending on $\params$ such that
  solutions $v\in\Xi_\ve$ cannot contact $\pl\Xi_\ve$ by achieving
  equality in condition~\eqref{vkplusL2} at any $\tau\geq\bt$.
\end{lemma}

\begin{proof}
  Arguing as in the proof of Lemma~\ref{ExitLemma}, one obtains
  \[
  \frac12\frac{\rd}{\rd\tau}\|v_{k+}\nh^2 =\lh v_{k+},\cA
  \vh+(\pl_\tau-\cA)\vh\rh.
  \]
  Here, one has $\lh v_{k+},\cA\vh\rh\leq-\mu_{k+1}\|v_{k+}\nh^2$,
  from which it follows that
  \[
  \frac12\frac{\rd}{\rd\tau}\|e^{\mu_k\tau}v_{k+}\nh^2
  \leq(\mu_k-\mu_{k+1})\|e^{\mu_k\tau}v_{k+}\nh^2 +|\lh
  e^{\mu_k\tau}v_{k+},e^{\mu_k\tau}(\pl_\tau-\cA )\vh\rh|.
  \]
  Combining this inequality with Lemma~\ref{ErrorSmall}, we see that
  if a solution were to contact $\pl\Xi_\ve$, then one would have
  \[
  \frac12\frac{\rd}{\rd\tau}\|e^{\mu_k\tau}v_{k+}\nh^2\Big|_{\|e^{\mu_k\tau}v_{k+}\nh=\ve}
  \leq\ve\big(\frac{\ve}{2}+C_2 e^{-\frac12\mu_k\tau}\big)<0
  \]
  for all $\tau\geq\bt\geq\tau_{10}$. By continuity, in a
  neighborhood of $\{\|e^{\mu_k\tau}v_{k+}\nh=\ve\}$, one has
  $\frac{d}{d\tau}\|e^{\mu_k\tau}v_{k+}\nh^2<0$ for $\tau\geq\bt$;
  this ensures that equality is never achieved.
\end{proof}

\medskip
Our final two entrapment results apply to solutions in originating in the
set $D_{\ve}^{\bt}$ defined by
\begin{equation}
  D_{\ve}^{\bt} :=
  \left\{v:
  \begin{array}{l}
  \|v_{k}(\cdot,\bt)-b_ke^{-\mu_k\bt}h_k\nh\leq\frac{\ve}{2}e^{-\mu_k\bt}\\ \\
  \|v_{k}(\cdot,\bt)-b_ke^{-\mu_k\bt}h_k\|_{C^1(2P)}\leq 100e^{-\mu_k\bt}
  \end{array}
  \right\}.
  \label{InitialDisc}
\end{equation}
This set is designed so that solutions satisfying
$v(\cdot,\bt)\in\Xi_\ve\cap D_{\ve}^{\bt}$ cannot exit $\Xi_\ve$ by violating
conditions~\eqref{vkL2} or \eqref{DefineParabolicPW}.

\begin{lemma}
  There exists $\bt\geq\tau_{10}$ large enough, depending on $\params$,
  such that solutions satisfying $v(\cdot,\bt)\in\Xi_\ve\cap D_{\ve}^{\bt}$
  cannot contact $\pl\Xi_\ve$ by achieving equality in condition~\eqref{vkL2}
  at any $\tau\geq\bt$.
\end{lemma}

\begin{proof}
  We first note that one can write
  $v_k(\sigma,\tau)=\phi(\tau)\eta(P^{-1}\sigma)h_k(\sigma)$, where
  the quantity
  \[
  \phi(\tau):=\|h_k\nh^{-2}\lh\vh,h_k\rh
  \]
  can be shown to evolve by
  \[
  \frac{\rd}{\rd \tau} \phi
    =-\mu_k\phi+\|h_k\nh^{-2}\lh(\pl_\tau-\cA)\vh,h_k\rh.
  \]
  Observing that
  \[
  \|e^{\mu_k\tau}v_k-b_kh_k\nh=(e^{\mu_k\tau}\eta\phi-b_k)\|h_k\nh,
  \]
  one computes that
  \begin{align*}
    \frac12\frac{\rd}{\rd\tau}\|e^{\mu_k\tau}v_k-b_kh_k\nh^2
    &=\|h_k\nh^2 (e^{\mu_k\tau}\phi-b_k) \frac{d}{d\tau}(e^{\mu_k\tau}\phi-b_k)\\
    &\leq\eta\|e^{\mu_k\tau}v_k-b_kh_k\nh |\lh(\pl_\tau-\cA )\vh,h_k\rh|.
  \end{align*}
  It then follows from Lemma~\ref{ErrorSmall} that one has
  \[
  \frac{\rd}{\rd\tau}\|e^{\mu_k\tau}v_k-b_kh_k\nh \leq C_2 \|h_k\nh
  e^{-\frac32\mu_k\tau}.
  \]
  So if $v\in\Xi_\ve\cap D_{\ve}^{\bt}$, then for all $\tau\geq\bt$,
  one has
  \[
  \|e^{\mu_k\tau}v_k-b_kh_k\nh \leq\frac{\ve}{2}+\frac{2C_2
    \|h_k\nh}{3\mu_k}e^{-\frac32\mu_k\bt} <\ve
  \]
  provided that $\bt$ is chosen sufficiently large, depending on
  $b_k$ and $\ve$ (which depends on $\params$ via Lemma~\ref{BrandenburgerTor}).
  We conclude that equality in \eqref{vkL2} cannot be attained.
\end{proof}

\begin{lemma}
\label{PWentrapment}
There exists $W$ sufficiently large, depending only on $b_k$, and there exists
$\bt\geq\tau_{10}$ sufficiently large, depending on $\params$, such that solutions
satisfying $v(\cdot,\bt)\in\Xi_\ve\cap D_{\ve}^{\bt}$ cannot contact $\pl\Xi_\ve$
by achieving equality in \eqref{DefineParabolicPW} at any $\tau\geq\bar\tau$.
\end{lemma}

\begin{proof}
Applying Lemma~\ref{StrapYourBoots} with $\tc=\bt\geq\tau_{10}$ and $M=100$ yields
\[
    \sup_{|\sigma|\leq P}\big\{|v(\sigma,\tau)| + |v_\sigma(\sigma,\tau)|\big\} \leq
    C_3 e^{-\mu_k\tau}
\]
for all $\tau\geq\bt$, where $C_3:=101 C_6$ depends only on $b_k$. Hence the result
holds for any $W>C_3$.
\end{proof}

\section{Constructing suitable initial data} \label{InitialData}

With the properties of the tubular neighborhood $\Xi_{\ve}$ surrounding a given
formal solution $\hat g_{\{n, k, b_k\}}(t)$ established (primarily in
Sections~\ref{PuttingUpBarriers}--\ref{ExitStrategy}), including the exit and
entrapment properties of portions of the boundary of $\Xi_{\ve}$, the remaining
work needed in order to complete the proof of our main theorem is to show that
there exist parameterized sets of initial data yielding Ricci flow solutions
that have the properties discussed in the introduction to Section~\ref{Structure}.
In particular, presuming that the choices of $\hat g_{\{n, k, b_k\}}(t)$ have been
fixed, we seek a continuous bijective map $\Phi$ from a closed topological $k$-ball
$\mathcal{B}^k$ to the set of rotationally symmetric metrics on
$\mathcal{S}^{n+1}$ with the following properties, for some choice of time
$t_0\in[0,T)$: (i) the image of $\Phi$ is contained in $\Xi_{\ve} \bigcap
\{t=t_0\}$; (ii) the image of $\Phi$ restricted to $\partial \mathcal{B}^k$ is
contained in the exit set $(\partial \Xi_{\ve})_{-}$; (iii) $\Phi |_{\partial
\mathcal{B}^k}$ is not contractible; and (iv) the image of $\Phi$ is far enough
away from the neutral part of the boundary $(\pl\Xi_\ve)_\circ$ to guarantee
that so long as a Ricci flow solution that starts from initial data in
$\Phi(\mathcal{B}^k)$ stays in $\Xi_{\ve}$, it never reaches
$(\pl\Xi_\ve)_\circ$.

We show in this section that for each choice of $\hat g_{\{n, k,b_k\}}(t)$
and corresponding $\Xi_{\ve}$, sets of initial data satisfying these
conditions can be found, if $\ve$ is small enough.

We first fix the initial time $t_0$ at which we choose our sets of
initial data. Noting that the results stated in the exit and
entrapment lemmas discussed in Section~\ref{InOrOut} are only
guaranteed to work for $\tau\geq\bar\tau$ (corresponding to
$t\geq T-e^{-\bar\tau}$), we set $t_0:=T-e^{-\bar\tau}$, where
$\bar\tau$ may be presumed to be large enough so that one has
$t_0>0$ and $\frac65 e^{\gk\bt/5}\geq2P$.

We next define initial data locally using coordinates adapted
to the parabolic region of the formal solution, since it is in
this region where the restrictions stated above are most critical.
Working with the metric variables and coordinates discussed
in Section~\ref{ParabolicCoordinates}, and recalling that $k\ge3$ and
$b_k<0$ are fixed, we see that for every set of $k$ constants
\[
b_j\in \big[-\underline{B}_j,\overline{B}_j\big],\qquad 0\leq j \leq
k-1,
\]
(with $\underline{B}_j$ and $\overline{B}_j$ small positive constants to be
determined below), if we write
\[
u=1+\sum_{j=0}^k b_j e^{-\mu_k\bt}h_k(\sigma),\qquad
\Big(|\sigma|\leq\frac75e^{\gk\bt/5}\Big),
\]
and if in addition we set $\vp(\cdot,\bar \tau)\equiv1$, then
a metric $g=(\rd\sigma)^2+u^2 g_{\rm can}$ is determined locally.
It is clear that one can choose $\underline{B}_j,\overline{B}_j>0$ small
enough such that the corresponding metrics are as close as desired in the
$C^1$ topology to the approximate solution $1+b_k e^{-\mu_k\bt}h_k(\sigma)$
in the region $|\sigma|\leq\frac75e^{\gk\bt/5}$. Closeness to $\hat g_{\{n,k, b_k\}}$
in that same region then follows.

With an eye toward conditions \eqref{vkminusL2}--\eqref{vkplusL2} and
\eqref{InitialDisc}, we choose $\underline{B}_j,\overline{B}_j>0$
so that the sharp inequalities
\begin{align*}
  \| v_{k-} \nh  &<\ve e^{-\mu_k\bt},\\
  \| v_{k}-b_ke^{-\mu_k\tau}h_k\nh &<\frac{\ve}{2} e^{-\mu_k\bt},\\
  \| v_{k+}\nh &<\ve e^{-\mu_k\bt},
\end{align*}
hold for
$(b_0,\dots,b_{k-1})\in\prod_{j=0}^{k-1}\big(-\underline{B}_j,\overline{B}_j\big)$,
with
\[
    \| v_{k-} \nh=\ve e^{-\mu_k\bt}
\]
holding on the boundary, which ensures by Lemma~\ref{ExitLemma} that solutions originating
from initial data on the boundary of the topological ball
$\prod_{j=0}^{k-1}\big[-\underline{B}_j, \overline{B}_j\big]$ immediately exit $\Xi_\ve$.
We verify that in choosing $\underline{B}_j$ and $\overline{B}_j$ as indicated, the conditions
for these initial data to lie within the tube (at least for the portion of the metrics
pertaining to the parabolic region) are satisfied. We may assume that $\ve$ is small
enough, hence that $\underline{B}_j,\overline{B}_j>0$ are small enough, so that the
$C^1$ condition in \eqref{InitialDisc} is satisfied. Combined with the second inequality
above, this ensures that the initial data constructed here belong to $D_{\ve}^{\bt}$.
In turn, by Lemma~\ref{PWentrapment}, this ensures that the data satisfy \eqref{DefineParabolicPW}.

We next obtain globally defined initial data in
$(\Xi_\ve\bigcap\{\tau=\bt\})\bigcap D_{\ve}^{\bt}$ by using cutoff
functions to smoothly glue the metric $g=(\rd\sigma)^2+u^2 g_{\rm can}$
defined locally above to the formal solution in the intermediate and tip
regions, and (for $k$ odd) to that portion of the manifold which remains
nonsingular. More specifically, we proceed as follows.

For $\frac65e^{\gk\bt/5}\leq\sigma\leq\frac75e^{\gk\bt/5}$, the
construction above ensures that $z=\psi_s^2$ is as close as desired to
its value in the formal solution~\eqref{eq:z-expansion-near-neck},
namely
\[
\hat z_{\{n,k,b_k\}}(u)=c_k e^{-2\gk\tau}(1-u)^{1+2\gk} \qquad
\big(1>u\gg e^{-\gk\tau}\big).
\]
In particular, making $\ve$ smaller if necessary, we can ensure that
each initial datum lies between the barriers $z_- < z_+$
constructed in Lemma~\ref{OrderIntermediate} for all $u$ where
both it and the barriers are defined. One may thus extend the solution
by means of a cutoff function, obtaining $\tilde z$ with the
properties that $\tilde z=z$ at $u\big(\frac65e^{\gk\bt/5}\sigma\big)$
and $\tilde z=\hat z_{\{n,k,b_k\}}$ for all
$0\leq u\leq u\big(\frac75e^{\gk\bt/5}\sigma\big)$. Here we use the fact that
$\hat z_{\{n,k,b_k\}}$ smoothly extends to $u=0$ using the construction
in Section~\ref{approximate-sol-at-tip}. Moreover, $\hat z_{\{n,k,b_k\}}$
lies between the barriers constructed in Lemma~\ref{OrderTip}.

If $k$ is even, we repeat this step for
$\frac65e^{\gk\bt/5}\leq-\sigma\leq\frac75e^{\gk\bt/5}$, and we are done.

If $k$ is odd and
$\frac65e^{\gk\bt/5}\leq-\sigma\leq\frac75e^{\gk\bt/5}$, we again find
that each initial datum $\tilde z$ is as close as desired to the formal
solution
\[
\hat z_{\{n,k,b_k\}}(u)= \hat c_k e^{-2\gk\tau}(u^2-1)u^{-2} \qquad
\big(1<u\ll e^{\tau/2}\big),
\]
hence lies between the barriers $\tilde z_-<\tilde z_+$ constructed
in Lemma~\ref{OrderOuter}.  So we again extend $\tilde z$ by
means of a cutoff function until $u=\hat\delta e^{\tau/2}$, namely
$\psi=\hat\delta$. There we smoothly glue on a smooth punctured sphere
as in Section~8.2 of \cite{AK04}.  We omit further details.

\end{document}